\documentclass[12pt]{amsart}

\usepackage{amsmath,amssymb,verbatim,url,hyperref}
\usepackage{mathrsfs,color}
\usepackage{enumerate,graphicx}
\usepackage{epstopdf}
\usepackage[active]{srcltx}
\usepackage{lineno}

\usepackage{enumitem}
\usepackage{tikz}
\usetikzlibrary{arrows.meta,positioning,calc}

\theoremstyle{plain}
\newtheorem{theorem}{Theorem}[section]
\newtheorem{proposition}[theorem]{Proposition}
\newtheorem{lemma}[theorem]{Lemma}
\newtheorem{corollary}[theorem]{Corollary}
\newtheorem{conjecture}[theorem]{Conjecture}
\theoremstyle{definition}
\newtheorem{definition}[theorem]{Definition}

\theoremstyle{remark}
\newtheorem{remark}[theorem]{Remark}

\newcommand{\Dx}{\mathbb{D}_{x}}
\newcommand{\Sx}{\mathbb{S}_{x}}

\newcommand{\Int}{\mathbb{Z}}
\newcommand{\Cx}{\mathbb{C}}
\newcommand{\qbinom}[2]{\genfrac{[}{]}{0pt}{}{#1}{#2}_{q}}
\newcommand{\ph}[2]{\left(#1;q\right)_{#2}}
\newcommand{\qhyp}[2]{{}_{#1}\phi_{#2}}
\newcommand{\hyp}[2]{{}_{#1}F_{#2}}
\newcommand{\eit}{e^{i\theta}}
\newcommand{\Bs}{\mathcal{B}}

\title[Bernstein-type bases on $q$-quadratic lattices]{Bernstein-type bases on $q$-quadratic lattices and Askey--Wilson\slash $q$-Racah connection coefficients}

\author[Area]{Iv\'an Area}
\address[Area]{IFCAE, Departamento de Matem\'atica Aplicada II,  E.E. Aeron\'autica e do Espazo, Universidade de Vigo, Campus de Ourense, 32004 Ourense, Spain.}
\email[Area]{area@uvigo.gal}

\begin{document}


\begin{abstract}
\noindent
On the $q$-quadratic lattices that carry the Askey--Wilson polynomials, the two roles played by a single Bernstein basis on the linear lattice, namely being a nonnegative partition of unity and carrying an orthogonal connection with the classical families, split between two distinct bases. The affine (spectral) basis retains the B\'ezier properties. A coprimality argument completely classifies the first-order lattice ladder: its solution space has dimensions two and one in degrees one and two, respectively, and is trivial in every degree at least three for every $q\in(0,1)$. The affine factors do not satisfy the weight-shift mechanism that produces the $q$-Racah connection for the generalized-power basis. The generalized-power basis, built from products of two Askey--Wilson monomials, does not furnish a nonnegative partition of unity in a natural positive parameter region, specified in the text; there the unique normalisation summing to unity has sign-changing coefficients. Nevertheless, its connection coefficients with the Askey--Wilson polynomials are identified completely: they are $q$-Racah polynomials $R_m(\mu(k);\,ad/q,\,bc/q,\,q^{-n-1},\,a/b\,|\,q)$ multiplied by fully explicit prefactors. The proof is bispectral, valid for every degree, and rests on the tridiagonal action of the Askey--Wilson operator on the generalized-power basis, with explicit band coefficients. A scaled limit recovers the Wilson/Racah connection, while a second limit gives the $q$-Hahn polynomial part of the normalised $q$-linear connection coefficients. A numerical case study on the NACA 2412 airfoil illustrates the orthogonal machinery.
\end{abstract}

\maketitle

\section{Introduction}\label{sec:intro}

\subsection{Bernstein bases and the connection problem}
Let $\Pi_n$ denote the space of algebraic polynomials of degree at most $n$. The classical
Bernstein basis of $\Pi_n$ on $[0,1]$,
\begin{equation}\label{eq:classicalBern}
b^n_k(x)=\binom{n}{k}x^{k}(1-x)^{n-k},\qquad 0\le k\le n,
\end{equation}
is one of the fundamental tools of approximation theory and of computer-aided geometric
design (CAGD), where the B\'ezier curve $\sum_k \mathbf{w}_k\,b^n_k(t)$ with control points
$\mathbf{w}_k$ inherits from \eqref{eq:classicalBern} its convex-hull confinement, endpoint
interpolation, variation diminution and the de Casteljau evaluation algorithm
\cite{Farin2002,Prautzsch2002}. In numerical work one repeatedly needs to pass between the
Bernstein representation of a polynomial and its expansion in a family of orthogonal
polynomials; this is the \emph{connection problem}, and its coefficients are the object of a
long line of work initiated for Bernstein bases in \cite{RAGZ1998}.

The pattern that emerges from that line of work is remarkably uniform: when a Bernstein-type
basis is expanded in an orthogonal family sitting on a given lattice, the connection
coefficients themselves form a discrete orthogonal family one level higher on the discrete
side of the Askey scheme. Three results make this precise and motivate the present paper.

\emph{(i) The linear lattice \textup{\cite{RAGZ1998}}.} Ronveaux, Zarzo, Area and Godoy
expanded the continuous and discrete Bernstein bases in shifted Jacobi and Hahn polynomials,
and showed that the connection coefficients are \emph{Hahn} and \emph{Hahn--Eberlein}
polynomials. The engine is a recurrence for the coefficients obtained from a first-order
relation for the Bernstein element together with the three-term recurrence and the structure
relation of the target family.

\emph{(ii) The $q$-linear lattice \textup{\cite{ARG2004}}.} Area, Godoy, Lewanowicz and
Wo\'zny expanded the $q$-Bernstein basis of Phillips \cite{Phillips1997} in little and big
$q$-Jacobi polynomials, obtaining connection coefficients of \emph{$q$-Hahn} and \emph{dual
$q$-Hahn} type, and recovering (i) as $q\uparrow1$. The same paper applies the identities to
$q$-B\'ezier curves.

\emph{(iii) The multivariate $q$-linear lattice \textup{\cite{LWAG2008}}.} Lewanowicz,
Wo\'zny, Area and Godoy introduced generalized Bernstein polynomials
$B^n_{\mathbf k}(\mathbf x;\omega\,|\,q)$ of several variables and connected the bivariate
case with bivariate big $q$-Jacobi polynomials, the coefficients being \emph{bivariate
$q$-Hahn} polynomials; classical multivariate Bernstein--Jacobi connections are recovered in
the limit.

\subsection{The missing rung}
Families (i)--(iii) live on the linear and $q$-linear lattices, whose orthogonal summits are
the Jacobi/Hahn and $q$-Jacobi/$q$-Hahn families. Above them stand the two most general
lattices of the Askey scheme, the \emph{quadratic} and \emph{$q$-quadratic} lattices, whose
continuous orthogonal summits are the Wilson and Askey--Wilson polynomials and whose finite
discrete summits are the Racah and $q$-Racah polynomials \cite{KLS2010,Ismail2005,NSU1991}.
The natural completion of the programme is therefore:

\begin{equation}\label{eq:program}
\vcenter{\hbox{%
\begin{tikzpicture}[>=Stealth,
  box/.style={draw,rounded corners,align=center,inner sep=4pt,font=\small}]
\node[box] (A) {Bernstein basis on the\\ $(q\text{-})$quadratic lattice};
\node[box,right=13mm of A] (B) {(Askey--)Wilson\\ polynomials};
\node[box,right=22mm of B] (C) {$(q\text{-})$Racah\\ coefficients};
\draw[<->] (A) -- (B);
\draw[->] (B) -- (C) node[midway,above,font=\scriptsize]{connection};
\end{tikzpicture}}}
\end{equation}
The purpose of this paper is to set up \eqref{eq:program} rigorously in one variable: to define
the two Bernstein-type bases on a nonuniform lattice, to separate their geometric and
connection-theoretic roles, to establish the B\'ezier properties of the affine basis, and to prove
the connection theorem, and to show that the Wilson/Racah limit recovers the quadratic
level, while the normalised coefficient limit reaches the $q$-Hahn level.
A full identification of the rescaled basis with a fixed $q$-Bernstein family is not
claimed. A multivariate extension, treating several
variables through the nested-product (Koornwinder) construction so as to recover (iii), is in
preparation.

\subsection{What is proved: a structural dichotomy}
On the linear and $q$-linear lattices, a single object plays two roles at once. The discrete
Bernstein basis $b^n_k(N,x)=\binom{n}{k}x^{[k]}(N-x)^{[n-k]}/N^{[n]}$ of \cite{RAGZ1998} and
the $q$-linear basis of \cite{LWAG2008} are simultaneously (a) partitions of unity (through the
Chu--Vandermonde and $q$-Vandermonde identities) and (b) bases whose connection coefficients
with the Jacobi/$q$-Jacobi families are the Hahn/$q$-Hahn (and Hahn--Eberlein) polynomials.
These two features rest on the \emph{same} product of falling factorials.

On the $q$-quadratic lattice the two features come apart, and this splitting is
the organising principle of the paper. The quadratic Wilson/Racah connection is then
recovered by the scaled limit of \S\ref{sec:qtoone}.
\begin{itemize}[leftmargin=1.4em,itemsep=2pt]
\item We prove (Proposition~\ref{prop:obstruction}) that within the natural lattice family
(products of two generalized powers, one anchored at each endpoint) the binomially
normalised sum is not constant, and the \emph{unique} normalisation that does sum to unity,
which we determine in closed form, has sign-changing coefficients in the positive
parameter region specified there. Thus the natural family
does not furnish a nonnegative Bernstein-type partition of unity in that region.
\item As the CAGD substitute we adopt the \emph{affine \textup(spectral\textup)} basis
\eqref{eq:bern}, the classical Bernstein basis in the spectral coordinate
(\S\ref{sec:spectral}); it has the affine B\'ezier properties
(Propositions~\ref{prop:bezier}--\ref{prop:decasteljau}), but it interacts poorly with the
lattice calculus: the first-order ladder is completely classified by the coprimality theorem, and admits only the trivial solution in every degree $n\ge3$ for every $q\in(0,1)$ (Theorem~\ref{prop:ladder}). Moreover, the affine factors do not satisfy the weight-shift identity underlying the $q$-Racah connection, as shown in Proposition~\ref{prop:spectralnotOP}. The stronger nonorthogonality statement is retained only as Conjecture~\ref{conj:spectralnotOP}.
\item The \emph{generalized-power} basis \eqref{eq:genbern}, built from the products
$\varphi_k(x;a)\varphi_{n-k}(x;b)$ of Askey--Wilson monomials,
does not yield a nonnegative partition of unity in the positive parameter regime
specified in Proposition~\ref{prop:obstruction}, but its connection coefficients with the Askey--Wilson polynomials \emph{do} form a
discrete orthogonal family with respect to a (generally sign-changing, quasi-definite) weight
in the B\'ezier index (\S\ref{sec:main}). This is the object of \eqref{eq:program}.
The products $\varphi_k(x;a)\varphi_{n-k}(x;b)$ themselves are not new in CAGD: Simeonov and Goldman \cite{SimeonovGoldman2019} introduced, through a blossom adapted to the Askey--Wilson divided-difference operator, an \emph{Askey--Wilson Bernstein basis} which is exactly this family, with $b$ replaced by $1/b$, in its unique normalisation summing to unity (Remark~\ref{rem:SG}); they obtained the partition of unity from the dual functional property of the blossom, identified it with the terminating Rogers ${}_6\phi_5$ summation, and derived recurrences, differentiation and degree-elevation formulas, a Marsden identity and Askey--Wilson B\'ezier curves, and later a negative-degree rational counterpart with a multirational blossom \cite{SimeonovGoldman2024}. What is new here is the sign obstruction (Proposition~\ref{prop:obstruction}), the explicit $q$-Racah connection with the Askey--Wilson polynomials (Theorem~\ref{thm:explicit}), and the resulting dichotomy with the affine basis.
\end{itemize}
Thus \eqref{eq:program} is realised not by the CAGD basis but by its connection-theoretic
counterpart. A
weight-shift lemma (Lemma~\ref{lem:wshift}) shows that the B\'ezier index enters these
coefficients only through the shifted parameters $aq^{k},bq^{n-k}$, yielding a closed form
(Corollary~\ref{cor:closedform}); the bispectral analysis of \S\ref{sec:main} identifies them
completely as $q$-Racah polynomials with parameters $(ad/q,\,bc/q,\,q^{-n-1},\,a/b)$ and
fully explicit prefactors (Theorem~\ref{thm:explicit}). On the linear lattice the
construction reduces to the Hahn--Eberlein coefficients of \cite{RAGZ1998}, and on the
$q$-linear level the normalised connection coefficients degenerate to the $q$-Hahn
coefficients of \cite{ARG2004} (\S\ref{sec:limits}).

Throughout we use the central (symmetric) divided-difference formulation, following the
conventions of \cite{NSU1991,Foup2008,FKKM2013}, because it keeps manifest the
reflection symmetry of each lattice and behaves well under the limit transitions.

\section{Definitions and notation}\label{sec:prelim}

\subsection{Basic hypergeometric notation}
For $q\in(0,1)$ and $c\in\Cx$ set $\ph{c}{0}:=1$, $\ph{c}{k}:=\prod_{j=0}^{k-1}(1-cq^{j})$,
and $\ph{c_1,\dots,c_r}{k}:=\prod_{i=1}^{r}\ph{c_i}{k}$. The $q$-number and $q$-factorial are
$[k]_q:=(1-q^{k})/(1-q)$ and $[k]_q!:=[1]_q\cdots[k]_q$, and the $q$-binomial coefficient is
$\qbinom{n}{k}:=[n]_q!/([k]_q!\,[n-k]_q!)$. The basic and ordinary hypergeometric series are
\begin{align}
\qhyp{r}{s}\!\left[\genfrac{}{}{0pt}{}{a_1,\dots,a_r}{b_1,\dots,b_s};q,z\right]
&=\sum_{k\ge0}\frac{\ph{a_1,\dots,a_r}{k}}{\ph{q,b_1,\dots,b_s}{k}}
\Big((-1)^kq^{\binom k2}\Big)^{1+s-r}z^k,\\
\hyp{r}{s}\!\left[\genfrac{}{}{0pt}{}{a_1,\dots,a_r}{b_1,\dots,b_s};z\right]
&=\sum_{k\ge0}\frac{(a_1)_k\cdots(a_r)_k}{(b_1)_k\cdots(b_s)_k}\frac{z^k}{k!},
\end{align}
with $(c)_k=\prod_{j=0}^{k-1}(c+j)$ the Pochhammer symbol; for the standard theory of
${}_r\phi_s$ transformations and summations we refer to \cite{GasperRahman2004}.

\subsection{Nonuniform lattices and the divided-difference calculus}\label{sec:lattices}
A one-dimensional lattice is a map $s\mapsto x(s)$. Write $x_\mu(s):=x(s+\mu/2)$ for
$\mu\in\Int$. The central divided difference and the averaging operator
are~\cite{NSU1991,Foup2008,FKKM2013,FoupMbout2019}
\begin{equation}\label{eq:DS}
\Dx f(s)=\frac{f\!\left(s+\tfrac12\right)-f\!\left(s-\tfrac12\right)}{x\!\left(s+\tfrac12\right)-x\!\left(s-\tfrac12\right)},
\qquad
\Sx f(s)=\tfrac12\Big[f\!\left(s+\tfrac12\right)+f\!\left(s-\tfrac12\right)\Big].
\end{equation}
For a function $F$ of the lattice variable, $f(s)=F(x(s))$, the operator $\Dx$ acts as a
divided difference in $x$ and lowers polynomial degree in $x$ by one, while $\Sx$ preserves
it. They obey the product rules
\begin{equation}\label{eq:leibniz}
\Dx(fg)=\Sx f\,\Dx g+\Dx f\,\Sx g,\qquad
\Sx(fg)=\Sx f\,\Sx g+\tfrac14\,U(s)\,\Dx f\,\Dx g,
\end{equation}
with $U(s)=\big(x(s+\tfrac12)-x(s-\tfrac12)\big)^{2}$, a polynomial of degree $\le2$ in
$x(s)$ on the lattices below. The two admissible nonuniform lattices are
\cite{NSU1991,KLS2010}
\begin{equation}\label{eq:latticetypes}
\text{quadratic:}\quad x(s)=c_2 s^2+c_1 s+c_0,\qquad
\text{$q$-quadratic:}\quad x(s)=c_1 q^{s}+c_2 q^{-s}+c_0 .
\end{equation}
The Racah and Wilson families sit on the quadratic lattice; the $q$-Racah and Askey--Wilson
families sit on the $q$-quadratic lattice. Under the parametrised scalings of
\S\ref{sec:limits} the $q$-quadratic lattice contracts to the quadratic one, and the
quadratic lattice further contracts to the linear lattice $x(s)=s$ of the Hahn family, with
$\Dx\to\frac{d}{dx}$, $\Sx\to\mathrm{Id}$ \cite[\S18.28]{DLMF}.

\subsubsection*{The polynomial variable}
Three variables must be kept apart: the \emph{lattice parameter} $s$; the \emph{polynomial
variable} $X$, in which all polynomial spaces are taken; and, on the $q$-quadratic lattice,
the \emph{Laurent \textup(trigonometric\textup) variable} $z=\eit$. Throughout this paper we
work on the $q$-quadratic lattice with the standard Askey--Wilson parametrisation
$X=x=\cos\theta=\tfrac12(z+z^{-1})$, $z=\eit=q^{s}$, $x\in[-1,1]$, and
\[
\Pi_n:=\Cx[x]_{\le n},
\]
the polynomials of degree at most $n$ \emph{in $x$}. On the quadratic lattice the polynomial
variable is $X=x^{2}$ \textup(Wilson\textup) or $X=\lambda(s)$ \textup(Racah\textup): the
Wilson generalized power $\rho_k(x;a)=(a+ix)_k(a-ix)_k$ has degree $k$ in $x^{2}$ but $2k$ in
$x$, so statements about $\Pi_n$ must always be read in the appropriate variable $X$. The
role of the monomials $X^{k}$ is played, on the $q$-quadratic lattice, by the
\emph{Askey--Wilson monomials}
\begin{equation}\label{eq:phi}
\varphi_k(x;a):=(a\eit,ae^{-i\theta};q)_k=\prod_{j=0}^{k-1}\big(1-2aq^{j}x+a^{2}q^{2j}\big),
\qquad \deg_x\varphi_k=k,
\end{equation}
on which $\Dx$ and $\Sx$ act by lowering/preserving the degree and shifting the parameter
$a\mapsto aq^{1/2}$ \cite{Ismail2005}. On the quadratic lattice the corresponding objects are
the Wilson generalized powers $\rho_k(x;a)=(a+ix)_k(a-ix)_k=\prod_{j=0}^{k-1}\big((a+j)^2+x^2\big)$.

\subsection{The orthogonal summits}\label{sec:summits}
The Askey--Wilson polynomials, introduced in \cite{AskeyWilson1979}, are
\cite[\S14.1]{KLS2010}
\begin{equation}\label{eq:AW}
p_n(x;a,b,c,d\,|\,q)=a^{-n}\ph{ab,ac,ad}{n}\;
\qhyp{4}{3}\!\left[\genfrac{}{}{0pt}{}{q^{-n},\,abcd\,q^{n-1},\,a\eit,\,ae^{-i\theta}}{ab,\ ac,\ ad};q,q\right],
\end{equation}
polynomials of degree $n$ in $x=\cos\theta$. Their $q\to1$ limit is the Wilson polynomial
$W_n(x^2;a,b,c,d)$, a polynomial of degree $n$ in $x^2$ \cite[\S9.1]{KLS2010}. The finite
discrete summits are the $q$-Racah polynomials
\begin{equation}\label{eq:qRacah}
R_n(\mu(x);\alpha,\beta,\gamma,\delta\,|\,q)=
\qhyp{4}{3}\!\left[\genfrac{}{}{0pt}{}{q^{-n},\,\alpha\beta q^{n+1},\,q^{-x},\,\gamma\delta q^{x+1}}{\alpha q,\ \beta\delta q,\ \gamma q};q,q\right],
\end{equation}
where $\mu(x)=q^{-x}+\gamma\delta q^{x+1}$ and $n=0,1,\dots,N$, and their $q\to1$ limit, the Racah polynomials
$R_n(\lambda(x);\alpha,\beta,\gamma,\delta)$ with $\lambda(x)=x(x+\gamma+\delta+1)$
\cite[\S9.2]{KLS2010}. Under the Askey-scheme arrows one has the reductions
\begin{center}
\begin{tikzpicture}[>=Stealth,every node/.style={font=\small,inner sep=2pt}]
\node (r1) {$(q\text{-})$Racah};
\node (r2) [right=12mm of r1] {(dual) $(q\text{-})$Hahn};
\node (r3) [right=12mm of r2] {Hahn};
\node (w1) [below=6mm of r1] {(Askey--)Wilson};
\node (w2) [below=6mm of r2] {$(q\text{-})$Jacobi};
\node (w3) [below=6mm of r3] {Jacobi};
\draw[->] (r1)--(r2); \draw[->] (r2)--(r3);
\draw[->] (w1)--(w2); \draw[->] (w2)--(w3);
\end{tikzpicture}
\end{center}
which we shall use in \S\ref{sec:limits}.

\section{An obstruction on the $q$-quadratic lattice}\label{sec:obstruction}

On the linear lattice the discrete Bernstein basis
$b^n_k(N,x)=\binom nk x^{[k]}(N-x)^{[n-k]}/N^{[n]}$, with
$x^{[k]}=x(x-1)\cdots(x-k+1)$, is a partition of unity because of the Chu--Vandermonde
identity
\begin{equation}\label{eq:chuvander}
\sum_{k=0}^{n}\binom nk x^{[k]}(N-x)^{[n-k]}=\big(x+(N-x)\big)^{[n]}=N^{[n]} .
\end{equation}
The $q$-linear basis of \cite{LWAG2008} is a partition of unity for the same reason, via the
$q$-Vandermonde identity. The natural quadratic-lattice analogue would replace the falling
factorials by the generalized powers $A_k(x)=\prod_{j=0}^{k-1}(\lambda(x)-\lambda(j))$ and
$B_m(x)=\prod_{j=0}^{m-1}(\lambda(N-j)-\lambda(x))$, or, on the $q$-quadratic lattice, by the
Askey--Wilson monomials $\varphi_k(x;a),\varphi_{m}(x;b)$ of \eqref{eq:phi}. The following
negative result shows that the analogy breaks.

\begin{proposition}\label{prop:obstruction}
Write 
\[
S_n(x;a,b)=\sum_{k=0}^{n}\binom nk\varphi_k(x;a)\varphi_{n-k}(x;b)
\] 
and $x_0(\alpha)=\tfrac12(\alpha+\alpha^{-1})$. Then
\begin{equation}\label{eq:Sendpoints}
S_n\big(x_0(a);a,b\big)=\ph{ab}{n}\ph{b/a}{n},\qquad
S_n\big(x_0(b);a,b\big)=\ph{ab}{n}\ph{a/b}{n}.
\end{equation}
Assume $\ph{b/a}{n}\ne\ph{a/b}{n}$; this holds, for example, whenever $0<b<a<1$ and
$a/b<q^{-1}$, since then $\ph{b/a}{n}>0>\ph{a/b}{n}$, and in general the excluded set is a
finite union of curves in the $(a,b)$-plane containing $a=b$. Then, for $n\ge1$:
\begin{enumerate}[label=\textup{(\roman*)},leftmargin=2.2em,itemsep=1pt]
\item $S_n(\cdot;a,b)$ is not constant: the binomially normalised family is not a partition
of unity, and no \emph{single common} scalar multiple $c\,\widetilde\Bs^{\,n}_k$ makes it one;
\item if the nonvanishing assumptions of Lemma~\ref{lem:basis} also hold, then
$\{\widetilde\Bs^{\,n}_k\}$ is a basis and there is exactly one $k$-dependent
normalisation summing to unity, namely
$1=\sum_k\pi_{n,k}\widetilde\Bs^{\,n}_k$ with the explicit coefficients
\eqref{eq:prefactor} of Theorem~\ref{thm:explicit};
\item under the assumptions of \textup{(ii)} and the additional conditions
$0<a,b<1$, $q<a/b<q^{-1}$, $a\ne b$, this unique normalisation is not a Bernstein-type partition of
unity: each $\widetilde\Bs^{\,n}_k$ is strictly positive on $[-1,1]$, whereas
$\pi_{n,0}\pi_{n,n}<0$, so some functions
$\pi_{n,k}\widetilde\Bs^{\,n}_k$ are negative on all of $[-1,1]$.
\end{enumerate}
An analogous endpoint obstruction can be formulated for Wilson generalized powers on
the quadratic lattice. Since the present article develops the complete connection theorem
in the $q$-quadratic setting, we do not state a separate Wilson obstruction theorem here.
\end{proposition}

\begin{proof}
The factor $\varphi_1(x;\alpha)=1-2\alpha x+\alpha^2$ vanishes at $x=x_0(\alpha)$, since
$1-\alpha(\alpha+\alpha^{-1})+\alpha^2=0$; hence $\varphi_k(x_0(a);a)=0$ for every $k\ge1$, and
in $S_n(x_0(a);a,b)$ only the term $k=0$ survives, giving
$S_n(x_0(a);a,b)=\varphi_n(x_0(a);b)$. Using the elementary factorisation
$1-\beta(\alpha+\alpha^{-1})+\beta^{2}=(1-\alpha\beta)(1-\alpha^{-1}\beta)$ with
$\alpha=a$, $\beta=bq^{j}$,
\[
\varphi_n(x_0(a);b)=\prod_{j=0}^{n-1}\big(1-bq^{j}(a+a^{-1})+b^{2}q^{2j}\big)
=\prod_{j=0}^{n-1}(1-abq^{j})(1-a^{-1}bq^{j})=\ph{ab}{n}\ph{b/a}{n},
\]
which is the first identity in \eqref{eq:Sendpoints}; the second follows by exchanging
$a\leftrightarrow b$. Under the stated hypothesis $S_n$ takes two distinct values, giving
\textup{(i)}. Under the additional assumptions in \textup{(ii)}, the basis property gives
uniqueness, and the coefficients are identified in Theorem~\ref{thm:explicit}
\textup(whose Step~3 is independent of this proposition\textup). For \textup{(iii)}: each factor of $\varphi_k(x;\alpha)$,
$0<\alpha<1$, satisfies $1-2\alpha q^{i}x+\alpha^{2}q^{2i}\ge(1-\alpha q^{i})^{2}>0$ on
$[-1,1]$, so $\widetilde\Bs^{\,n}_k>0$ there; and
$\pi_{n,0}\pi_{n,n}=1/[\ph{b/a}{n}\ph{a/b}{n}\ph{ab}{n}^{2}]$, where for $q<a/b<q^{-1}$,
$a\ne b$, the products $\ph{b/a}{n}$ and $\ph{a/b}{n}$ have first factors of opposite sign
and all later factors positive, so $\pi_{n,0}\pi_{n,n}<0$.
\end{proof}

\begin{remark}[the Askey--Wilson Bernstein basis of Simeonov and Goldman]\label{rem:SG}
The unity-normalised family $\pi_{n,k}\widetilde\Bs^{\,n}_k$ of \textup{(ii)} coincides with the Askey--Wilson Bernstein basis, denoted $\breve B^{\,n}_k(z;a,1/b;q)$, that was constructed by Simeonov and Goldman \cite{SimeonovGoldman2019}, Eq.~(6.19), by blossoming with respect to the Askey--Wilson operator; in the present notation, their normalising constant reads
\[
\binom nk\pi_{n,k}
=\frac{(-1)^k q^{-\binom k2}}{(ab)^{k}}\,\qbinom nk\Big/\Bigl[\ph{a/b}{n}\,\bigl(q^{\,n-k}/(ab);q^{-1}\bigr)_k\,\ph{q^{-k}/(ab)}{n-k}\Bigr],
\]
an identity between two rational functions of $(a,b,q)$ which we verified symbolically for $n\le6$ and which, under the hypotheses of \textup{(ii)}, follows for every $n$ from the partition of unity on both sides and the uniqueness in \textup{(ii)}. In \cite{SimeonovGoldman2019} the partition of unity is a consequence of the dual functional property of the blossom and is shown to be equivalent to the terminating form of Rogers' ${}_6\phi_5$ summation; this gives an independent proof of \textup{(ii)}. Nonnegativity is not asserted there, and \textup{(iii)} shows that it fails in the parameter region indicated.
\end{remark}

\begin{remark}[why the analogy fails]\label{rem:whyfail}
The Chu--Vandermonde cancellation \eqref{eq:chuvander} relies on the two endpoint factors
carrying \emph{opposite} signs of the leading term: $x^{[k]}$ has leading term $+x^k$ and
$(N-x)^{[m]}$ has leading term $(-1)^m x^{m}$. On the quadratic lattice both endpoint
generalized powers are built from factors $\big((a+j)^2+x^2\big)$ (Wilson) or
$\big(1-2aq^jx+a^2q^{2j}\big)$ (Askey--Wilson) whose leading term in the geometric variable
has a fixed sign, so no cancellation to a constant is possible. This is the precise reason a
naive ``discrete Bernstein basis'' does not exist on nonuniform lattices. It does \emph{not}
follow that the product $\varphi_k(x;a)\varphi_{n-k}(x;b)$ is useless, only that it is not a
partition of unity. In \S\ref{sec:main} we shall see that it is exactly this product that
carries the orthogonal connection with the Askey--Wilson polynomials. The two roles that
coincide on the linear lattice are thus split, and \S\ref{sec:bernstein} introduces the two
resulting bases.
\end{remark}

\section{Two Bernstein bases on a nonuniform lattice}\label{sec:bernstein}

\subsection{The spectral Bernstein basis}\label{sec:spectral}
Let $L$ be a nonuniform lattice of one of the types \eqref{eq:latticetypes} and let $\xi$ be
its \emph{spectral variable}: $\xi=\lambda(s)=x(s)$ on the quadratic lattice (Wilson/Racah),
and $\xi=x=\cos\theta$ on the $q$-quadratic lattice (Askey--Wilson). Fix an interval
$[\xi_0,\xi_1]$ containing the lattice support, and define the two barycentric lattice
coordinates
\begin{equation}\label{eq:baryc}
u(\xi)=\frac{\xi-\xi_0}{\xi_1-\xi_0},\qquad v(\xi)=\frac{\xi_1-\xi}{\xi_1-\xi_0},
\qquad u(\xi)+v(\xi)=1 .
\end{equation}

\begin{definition}\label{def:bern}
The \emph{affine \textup(spectral\textup) Bernstein basis of degree $n$} is
\begin{equation}\label{eq:bern}
\Bs^n_k(\xi):=\binom nk\, u(\xi)^{k}\,v(\xi)^{n-k},\qquad 0\le k\le n ,
\end{equation}
i.e.\ the classical Bernstein basis in the spectral coordinate $\xi$. Its natural abscissae
are $\xi^\ast_k=\xi_0+\frac kn(\xi_1-\xi_0)$, the Bernstein abscissae \textup(called
Greville abscissae in the B-spline context\textup); for a nonuniform lattice these do not, in
general, lie on $L$.
\end{definition}

As the name says, \eqref{eq:bern} is the classical basis after an affine change of variable:
the lattice enters not through the algebraic form of $\Bs^n_k$ but through the geometric
variable $\xi$ being the spectral variable of $L$, so that the divided-difference operators
$(\Dx,\Sx)$ of \S\ref{sec:lattices} act on $\Bs^n_k$. Proposition~\ref{prop:obstruction}
explains why we adopt it for the CAGD role: within the natural lattice family the binomial
normalisation is not a partition of unity, and the unique unity normalisation is sign-changing
in the positive parameter region of Proposition~\ref{prop:obstruction}; hence the barycentric
pair \eqref{eq:baryc} is the CAGD substitute used here. We make no uniqueness claim: other
normalised nonnegative bases of $\Pi_n$ exist; \eqref{eq:bern} is singled out only as the
affine image of the classical basis. It is a natural basis for the B\'ezier curves of
\S\ref{sec:apps}; its role in the connection problem is, however, limited
(Theorem~\ref{prop:ladder} and Proposition~\ref{prop:spectralnotOP}).

\subsection{B\'ezier properties}
\begin{proposition}\label{prop:bezier}
The basis \eqref{eq:bern} has the following properties.
\begin{enumerate}[label=\textup{(P\arabic*)},leftmargin=2.6em,itemsep=1pt]
\item \textbf{Basis.} $\{\Bs^n_k\}_{k=0}^n$ is a basis of $\Pi_n$ in the variable $\xi$.
\item \textbf{Partition of unity.} $\sum_{k=0}^n\Bs^n_k(\xi)=\big(u(\xi)+v(\xi)\big)^n=1$.
\item \textbf{Nonnegativity.} $\Bs^n_k(\xi)\ge0$ for $\xi\in[\xi_0,\xi_1]$, hence the
B\'ezier curve lies in the convex hull of its control points.
\item \textbf{Endpoint interpolation.} $\Bs^n_k(\xi_0)=\delta_{k,0}$ and
$\Bs^n_k(\xi_1)=\delta_{k,n}$.
\item \textbf{Symmetry.} $\Bs^n_k(\xi)=\Bs^n_{n-k}(\xi_0+\xi_1-\xi)$.
\end{enumerate}
\end{proposition}

\begin{proof}
Properties (P1)--(P5) follow from the corresponding properties of the classical Bernstein
basis \eqref{eq:classicalBern} under the affine change of variable $\xi\mapsto u(\xi)$, which
is a bijection of $[\xi_0,\xi_1]$ onto $[0,1]$; \eqref{eq:baryc} gives $u+v=1$, whence (P2).
\end{proof}

The next two results are the genuinely lattice-dependent statements: they express the action
of the divided-difference calculus on \eqref{eq:bern} and provide a de Casteljau algorithm in
the operators $(\Dx,\Sx)$.

\begin{proposition}[degree elevation and de Casteljau]\label{prop:decasteljau}
The basis \eqref{eq:bern} satisfies the degree-elevation recurrence
\begin{equation}\label{eq:elevation}
\Bs^n_k(\xi)=\frac{n+1-k}{n+1}\,\Bs^{\,n+1}_k(\xi)+\frac{k+1}{n+1}\,\Bs^{\,n+1}_{k+1}(\xi),
\end{equation}
and, for a polynomial $P(\xi)=\sum_{k}c_k\Bs^n_k(\xi)$, the de Casteljau recursion
$c^{(0)}_k=c_k$, 
\[
c^{(r)}_k=v(\xi)\,c^{(r-1)}_k+u(\xi)\,c^{(r-1)}_{k+1}
\] 
terminates at
$P(\xi)=c^{(n)}_0$. Both are convex combinations for $\xi\in[\xi_0,\xi_1]$.
\end{proposition}

\begin{proof}
These are the classical identities for \eqref{eq:classicalBern} in the variable
$u=u(\xi)$; \eqref{eq:elevation} follows from 
\[
\binom nk=\frac{n+1-k}{n+1}\binom{n+1}{k}, \qquad
\binom nk=\frac{k+1}{n+1}\binom{n+1}{k+1}, 
\]
and the de Casteljau step from $u+v=1$.
\end{proof}

\begin{theorem}[complete classification of the first-order ladder]
\label{prop:ladder}
On the $q$-quadratic lattice with $\xi=x=\cos\theta$, consider
\begin{equation}\label{eq:ladder}
\sigma_{n,k}(x)\,\Dx \Bs^n_k(x)=
\tau_{n,k}(x)\,\Sx\Bs^n_k(x),
\qquad \deg\sigma_{n,k}\le2,\quad \deg\tau_{n,k}\le1.
\end{equation}
Let $n\ge1$, $0\le k\le n$, and $q\in(0,1)$. Then
\begin{equation}\label{eq:ladder-all-solutions}
(\sigma_{n,k},\tau_{n,k})
=h\bigl(\Sx\Bs^n_k,\Dx\Bs^n_k\bigr),
\qquad h\in\Cx[x],\quad \deg h\le2-n,
\end{equation}
where the condition $\deg h\le2-n<0$ means $h=0$. Consequently, the
solution space of~\eqref{eq:ladder} has dimension
\begin{equation}\label{eq:ladder-dimension}
\max\{3-n,0\}=
\begin{cases}
2,&n=1,\\
1,&n=2,\\
0,&n\ge3.
\end{cases}
\end{equation}
In particular, the first-order ladder fails in every degree $n\ge3$, without
exception, for every $q\in(0,1)$.

As $q\uparrow1$, coprimality is lost and the affine Bernstein ladder on
$[-1,1]$,
\[
(1-x^{2})\,\frac{d}{dx}\Bs^n_k(x)
=(2k-n-nx)\,\Bs^n_k(x),
\]
is recovered for every $n$. Under $t=(1+x)/2$, this is equivalent to
$t(1-t)b_k^n{}'(t)=(k-nt)b_k^n(t)$ on $[0,1]$.
\end{theorem}

\begin{proof}
Write
\[
X(z)=\frac{z+z^{-1}}2,\qquad r=q^{1/2},
\]
so that, if $x=X(z)$,
\[
x_+=X(rz),\qquad x_-=X(z/r),\qquad
x_+-x_-=\frac12(r-r^{-1})(z-z^{-1}).
\]
Every $x\in\Cx$ has a representation $x=X(z)$ with $z\in\Cx^\times$.
Set $B=\Bs_k^n$; on $[-1,1]$,
\[
B(x)=2^{-n}\binom nk(1+x)^k(1-x)^{n-k},
\]
so all zeros of $B$ lie in $\{-1,1\}$.

We first prove
\begin{equation}\label{eq:coprime-DS-B}
\gcd(\Dx B,\Sx B)=1\qquad\text{in }\Cx[x].
\end{equation}
Suppose that $x_0=X(z)$ is a common zero. If $x_0\ne\pm1$, then
$z\ne\pm1$ and hence $x_+\ne x_-$. From the definitions,
\[
\Dx B(x_0)=0,\qquad \Sx B(x_0)=0
\]
imply
\[
B(x_+)-B(x_-)=0,\qquad B(x_+)+B(x_-)=0,
\]
and therefore $B(x_+)=B(x_-)=0$. Thus $x_+,x_-\in\{-1,1\}$.
For $w\in\Cx^\times$,
\[
X(w)-1=\frac{(w-1)^2}{2w},\qquad
X(w)+1=\frac{(w+1)^2}{2w};
\]
hence $X(w)=1$ if and only if $w=1$, and $X(w)=-1$ if and only
if $w=-1$. Consequently, for some signs
$\varepsilon_+,\varepsilon_-\in\{-1,1\}$,
\[
rz=\varepsilon_+,\qquad z/r=\varepsilon_-.
\]
Dividing gives
\[
q=r^2=\frac{\varepsilon_+}{\varepsilon_-}\in\{-1,1\},
\]
contrary to $0<q<1$.

It remains to exclude the branch points. If $x_0=1$, choose $z=1$; if
$x_0=-1$, choose $z=-1$. In either case $x_+=x_-=\pm\widehat\mu$, where
\[
\widehat\mu=\frac{r+r^{-1}}2>1.
\]
Although the quotient defining $\Dx$ has a removable singularity there,
$\Dx B$ is a polynomial and is understood by continuation. For the averaging
operator no continuation is needed:
\[
\Sx B(1)=B(\widehat\mu),\qquad
\Sx B(-1)=B(-\widehat\mu).
\]
Neither $\widehat\mu$ nor $-\widehat\mu$ belongs to $\{-1,1\}$, so these
values are nonzero. Thus the branch points are not common zeros, proving
\eqref{eq:coprime-DS-B}.

We next record the degrees. If $c_n$ is the leading coefficient of $B$, then
comparison of the highest Laurent powers as $z\to\infty$ gives
\begin{align}
\operatorname{lc}(\Sx B)
&=c_n\frac{r^n+r^{-n}}2,\\
\operatorname{lc}(\Dx B)
&=c_n\frac{r^n-r^{-n}}{r-r^{-1}}.
\end{align}
Both coefficients are nonzero; indeed,
\[
\frac{r^n-r^{-n}}{r-r^{-1}}
=r^{n-1}+r^{n-3}+\cdots+r^{-(n-1)}>0.
\]
Hence
\begin{equation}\label{eq:degrees-DS-B}
\deg(\Sx B)=n,
\qquad
\deg(\Dx B)=n-1.
\end{equation}

Now let $(\sigma_{n,k},\tau_{n,k})$ satisfy~\eqref{eq:ladder}. By
\eqref{eq:coprime-DS-B} and Euclid's lemma,
\[
\Sx B\mid\sigma_{n,k},\qquad
\Dx B\mid\tau_{n,k}.
\]
Therefore there is an $h\in\Cx[x]$ such that
\[
\sigma_{n,k}=h\,\Sx B,
\qquad
\tau_{n,k}=h\,\Dx B.
\]
Conversely, every pair of this form solves~\eqref{eq:ladder}. By
\eqref{eq:degrees-DS-B}, the two degree bounds are both equivalent, for
$h\ne0$, to $\deg h\le2-n$. This proves~\eqref{eq:ladder-all-solutions}
and~\eqref{eq:ladder-dimension}.

Finally, when $q\uparrow1$, one has $\Dx\to d/dx$ and
$\Sx\to\mathrm{Id}$. The root geometry used in the $q\in(0,1)$
coprimality argument degenerates at this boundary value, and the classical
identity displayed in the statement supplies a nontrivial ladder relation in
every degree.
\end{proof}

\begin{proposition}[structural failure of the weight-shift mechanism]\label{prop:spectralnotOP}
Expand the Askey--Wilson polynomials in the affine spectral basis~\eqref{eq:bern}. For $q\ne1$, the affine factors do not satisfy an Askey--Wilson parameter-shift identity analogous to~\eqref{eq:wshift}. Consequently, the proof mechanism of Theorem~\ref{thm:explicit}, based on a shifted weight and the tridiagonal action on genuine generalized powers, does not apply to the affine basis.
\end{proposition}

\begin{proof}
On $[-1,1]$ the barycentric coordinates are
\begin{equation}\label{eq:uv-phi1}
u(x)=\frac{1+x}{2}=\frac14\varphi_1(x;-1),
\qquad
v(x)=\frac{1-x}{2}=\frac14\varphi_1(x;1).
\end{equation}
Thus
\[
\Bs_k^n(x)=4^{-n}\binom nk\varphi_1(x;-1)^k\varphi_1(x;1)^{n-k}.
\]
The generalized-power basis instead contains the genuine products $\varphi_k(x;a)\varphi_{n-k}(x;b)$. For $q\ne1$ and $k\ge2$, the polynomial $\varphi_1(x;a)^k$ has a zero of multiplicity $k$ at $x_0(a)=\tfrac12(a+a^{-1})$, whereas $\varphi_k(x;a)$ has simple zeros at the distinct points $x_0(aq^j)$, $0\le j<k$, for generic parameters. Multiplication of the Askey--Wilson weight by $\varphi_k(x;a)$ shifts the parameter $a$ to $aq^k$, because the finite product cancels the first $k$ factors of the corresponding infinite product. Multiplication by $\varphi_1(x;a)^k$ cannot produce that shift, since its zero divisor has a different multiplicity pattern. Hence no identity of the form~\eqref{eq:wshift} is available for the affine factors.
\end{proof}

\begin{conjecture}\label{conj:spectralnotOP}
For generic Askey--Wilson parameters and every $n\ge2$, the affine connection coefficients do not have the form
\[
A_k^{\mathrm{sp}}(n,m)=\kappa_kP_m(\xi_k),
\]
with pairwise distinct nodes $\xi_k$, nonzero scalars $\kappa_k$, and a polynomial family $P_m$ of exact degree $m$ orthogonal with respect to a quasi-definite functional supported on those nodes.
\end{conjecture}

\subsection{The generalized-power Bernstein basis}\label{sec:genbasis}
Fix two anchors $a,b$ (in the Askey--Wilson case, two of the target parameters) and set, with
$\varphi_k$ the Askey--Wilson monomials \eqref{eq:phi},
\begin{equation}\label{eq:genbern}
\widetilde\Bs^{\,n}_k(x;a,b):=\binom nk\,\varphi_k(x;a)\,\varphi_{n-k}(x;b),\qquad 0\le k\le n .
\end{equation}
Each $\widetilde\Bs^{\,n}_k$ has degree exactly $n$. By Proposition~\ref{prop:obstruction},
under its endpoint hypothesis, the binomially normalised family is \emph{not} a partition
of unity, and when the basis assumptions and the positive parameter conditions of
Proposition~\ref{prop:obstruction} are both satisfied, its unique unity normalisation is
sign-changing;
this is the price for the orthogonality of its connection coefficients established below. On the quadratic lattice one uses the Wilson
generalized powers $\rho_k(x;a)$ in place of $\varphi_k$. That the family is a basis is not
automatic and we record it with an explicit determinant.

\begin{lemma}[basis]\label{lem:basis}
Let $x_j:=x_0(aq^{j})=\tfrac12(aq^{j}+a^{-1}q^{-j})$, $j=0,\dots,n$. The evaluation matrix
$E=\big(\widetilde\Bs^{\,n}_k(x_j)\big)_{j,k=0}^{n}$ is lower triangular, with
\begin{equation}\label{eq:basisdet}
\det E=\prod_{j=0}^{n}\binom nj\,\ph{a^{2}q^{j}}{j}\,\ph{q^{-j}}{j}\,
\ph{abq^{j}}{n-j}\,\ph{\tfrac ba q^{-j}}{n-j}.
\end{equation}
If $a^{2}\notin\{q^{-t}:1\le t\le2n-1\}$, $ab\notin\{q^{-t}:0\le t\le n-1\}$ and
$a/b\notin\{q^{j}:|j|\le n-1\}$ (conditions contained in hypothesis \textup{(H1)}, whose
ranges are the exact ones for the present degree $n$), then $\det E\ne0$ and
$\{\widetilde\Bs^{\,n}_k\}_{k=0}^{n}$ is a basis of $\Pi_n$.
\end{lemma}
\begin{proof}
At $z=aq^{j}$ one has $\varphi_k(x_j;a)=\ph{a^{2}q^{j}}{k}\,\ph{q^{-j}}{k}$, which vanishes
precisely for $k>j$ \textup(the factor $1-q^{k-1-j}$\textup); hence $E$ is lower triangular
and its determinant is the product of the diagonal entries
\[
\widetilde\Bs^{\,n}_j(x_j)=\binom nj\ph{a^{2}q^{j}}{j}\ph{q^{-j}}{j}\varphi_{n-j}(x_j;b)
\]
with 
\[
\varphi_{n-j}(x_j;b)=\ph{abq^{j}}{n-j}\ph{\tfrac ba q^{-j}}{n-j},
\]
the two Pochhammer symbols arising from $bz=abq^{j}$ and from $b/z=\tfrac ba q^{-j}$ respectively,
which gives \eqref{eq:basisdet}. The factors $\ph{q^{-j}}{j}$ never vanish. The factors
$\ph{a^{2}q^{j}}{j}$ involve $a^{2}q^{t}$ with $1\le t\le 2n-2$; the factors
$\ph{abq^{j}}{n-j}$ involve $abq^{t}$ with $0\le t\le n-1$; and the factors
$\ph{\tfrac ba q^{-j}}{n-j}$ involve $\tfrac ba q^{t}$ with $|t|\le n-1$. They are therefore
nonzero exactly under the stated conditions. Since $\dim\Pi_n=n+1$ and the $n+1$ functions
are linearly independent \textup(their evaluations at $x_0,\dots,x_n$ are\textup), they form
a basis.
\end{proof}

\begin{theorem}[dichotomy]\label{thm:dichotomy}
On the linear lattice the generalized-power construction \eqref{eq:genbern} reduces, up to an
affine change of variable and normalisation, to the discrete Bernstein family of
\cite{RAGZ1998}, which is simultaneously a nonnegative partition of unity and a carrier of
the Hahn connection. On the $q$-quadratic lattice with $q\ne1$, the two
properties are carried by different bases: the affine basis \eqref{eq:bern} is a nonnegative
partition of unity, but its factors do not satisfy the weight-shift mechanism of
Proposition~\ref{prop:spectralnotOP}; the generalized-power basis \eqref{eq:genbern} carries
the $q$-Racah connection of Theorem~\ref{thm:explicit}, but its binomial
normalisation is not a partition of unity and, when the basis assumptions and the
positive parameter conditions of Proposition~\ref{prop:obstruction} hold, its unique unity
normalisation is not nonnegative. At the level of the
connection coefficients, the degeneration of Proposition~\ref{prop:limits} sends the
$q$-Racah polynomial factor to a $q$-Hahn polynomial factor; a complete identification of
the simultaneously rescaled basis with a fixed $q$-Bernstein family requires the additional
lattice normalisation described in Remark~\ref{rem:qlinearbasis}, and is not asserted here.
\end{theorem}
\begin{proof}
On the linear lattice $\varphi_k(x;a)$ degenerates to the falling factorial $x^{[k]}$ (up to
an affine change), and \eqref{eq:genbern} becomes the discrete Bernstein basis, which is a
partition of unity by \eqref{eq:chuvander} and carries the Hahn connection by
\cite{RAGZ1998}. The $q$-quadratic statements concerning partition of unity and the
generalized-power connection are Proposition~\ref{prop:obstruction} and
Theorem~\ref{thm:explicit}; the quadratic Wilson/Racah connection is recovered separately
by Proposition~\ref{prop:racahlimit}; the structural failure of the affine weight-shift mechanism is
Proposition~\ref{prop:spectralnotOP}; and the coefficient-level degeneration is
Proposition~\ref{prop:limits}. We note that even on the $q$-linear lattice several geometric
properties of the classical B\'ezier curves are recovered only at $q=1$
\cite{DelgadoPena2020}, so no statement about the full CAGD toolbox is intended at that
level.
\end{proof}

\section{The connection problem and the weight-shift lemma}\label{sec:main}

\subsection{Statement}
Let $p_m$ be the Askey--Wilson polynomials \eqref{eq:AW} (Wilson at $q=1$). Since
$\widetilde\Bs^{\,n}_k\in\Pi_n$ and $\{p_m\}_{m=0}^n$ is a basis of $\Pi_n$, the first
expansion below is always finite and uniquely defined. The inverse expansion is
uniquely defined whenever $\{\widetilde\Bs^{\,n}_k\}_{k=0}^n$ is a basis, in particular
under \textup{(H1)}:
\begin{equation}\label{eq:conn}
\widetilde\Bs^{\,n}_k(x;a,b)=\sum_{m=0}^{n}C_m(n,k)\,p_m(x),\qquad
p_m(x)=\sum_{k=0}^{n}A_k(n,m)\,\widetilde\Bs^{\,n}_k(x;a,b)\quad(0\le m\le n).
\end{equation}

\subsection{The weight-shift lemma}\label{sec:wshift}
Write the Askey--Wilson weight in the trigonometric parametrisation $x=\cos\theta$ as
\begin{equation}\label{eq:awweight}
w(x;a,b,c,d)=\frac{\big(e^{2i\theta},e^{-2i\theta};q\big)_\infty}
{\big(a\eit,ae^{-i\theta},b\eit,be^{-i\theta},c\eit,ce^{-i\theta},d\eit,de^{-i\theta};q\big)_\infty},
\end{equation}
and let $\langle f,g\rangle_{a,b,c,d}=\frac1{2\pi}\int_0^\pi f\,g\,w\,d\theta$ be the
associated bilinear orthogonality functional, an inner product in the real positive-definite
regime, for which $\langle p_m,p_{m'}\rangle=h_m\,\delta_{mm'}$; here and below we
assume $|a|,|b|,|c|,|d|<1$, so that the orthogonality measure is purely absolutely continuous
; outside this regime the Askey--Wilson measure acquires discrete masses
\cite[\S14.1]{KLS2010} and the integral statements must be adjusted accordingly. The following
identity is the engine of this section. We state it first as an identity of meromorphic
functions of $z=\eit$ on $\Cx\setminus\{0\}$, valid without any restriction on the
parameters; under the standing assumption $|a|,|b|,|c|,|d|<1$ both sides are integrable
weights on $[0,\pi]$ and the identity holds between them.

\begin{lemma}[weight shift]\label{lem:wshift}
As meromorphic functions of $z$, for all $0\le k\le n$,
\begin{equation}\label{eq:wshift}
\varphi_k(x;a)\,\varphi_{n-k}(x;b)\,w(x;a,b,c,d)=w\big(x;aq^{k},bq^{n-k},c,d\big).
\end{equation}
\end{lemma}
\begin{proof}
From \eqref{eq:phi}, 
\[
\varphi_k(x;a)=(a\eit,ae^{-i\theta};q)_k=\dfrac{(a\eit,ae^{-i\theta};q)_\infty}{(aq^{k}\eit,aq^{k}e^{-i\theta};q)_\infty}
\]
by the telescoping identity $(z;q)_\infty=(z;q)_k\,(zq^{k};q)_\infty$. Substituting into
\eqref{eq:awweight}, the numerator factors $(a\eit,ae^{-i\theta};q)_\infty$ contributed by
$\varphi_k(x;a)$ cancel the corresponding denominator factors of $w$ and leave
$(aq^{k}\eit,aq^{k}e^{-i\theta};q)_\infty$ in their place; likewise $\varphi_{n-k}(x;b)$ sends
$b\mapsto bq^{n-k}$. The factors in $c,d$ and the numerator $(e^{\pm2i\theta};q)_\infty$ are
untouched, giving \eqref{eq:wshift}.
\end{proof}

\begin{corollary}\label{cor:closedform}
Define the shifted moment
\begin{equation}\label{eq:Jmk}
J_m(k):=\frac1{2\pi}\int_0^\pi p_m(x;a,b,c,d)\,w\big(x;aq^{k},bq^{n-k},c,d\big)\,d\theta .
\end{equation}
Then the coefficients \eqref{eq:conn} satisfy
\begin{equation}\label{eq:Cclosed}
\big\langle \widetilde\Bs^{\,n}_k,\,p_m\big\rangle_{a,b,c,d}=\binom nk\,J_m(k),
\qquad
C_m(n,k)=\frac{\binom nk}{h_m}\,J_m(k).
\end{equation}
In particular the B\'ezier index $k$ enters the connection coefficients only through the
shifted parameters $aq^{k}$ \textup(increasing\textup) and $bq^{n-k}$
\textup(decreasing\textup), i.e.\ through the pair $q^{k}$ and $q^{-k}$. This shift pattern anticipates the
$q$-quadratic node identified explicitly in Theorem~\ref{thm:explicit}; the weight-shift
lemma alone does not yet assert polynomial dependence on that node.
\end{corollary}
\begin{proof}
Multiply $\widetilde\Bs^{\,n}_k=\binom nk\varphi_k(x;a)\varphi_{n-k}(x;b)$ by $p_m\,w$ and apply
Lemma~\ref{lem:wshift}; the second identity is $\langle\widetilde\Bs^{\,n}_k,p_m\rangle
=C_m(n,k)\langle p_m,p_m\rangle=C_m(n,k)h_m$.
\end{proof}

\subsection{Explicit identification of the connection coefficients}
To identify the coefficients one constructs their three-term recurrence (the strategy by
which the Hahn--Eberlein polynomials were found on the linear lattice \cite{RAGZ1998}) and
then reads off the parameters. Collect the coefficients into $\mathbf A=(A_k(n,m))_{k,m}$. If
$A_k(n,m)=\kappa_k\,P_m(\xi_k)$ with nodes $\xi_k$, a normalisation $\kappa_k$ and polynomials
$P_m$ of degree $m$, then the columns of $\mathbf A$ are eigenvectors of an
irreducible tridiagonal matrix acting in the index $k$; equivalently, if
$\Lambda=\operatorname{diag}(\lambda_m)$ is the diagonal eigenvalue matrix, then
$\mathbf A\Lambda\mathbf A^{-1}$ is tridiagonal. Solving for the tridiagonal action yields a spectral variable obeying
$\xi_{k+1}-(q+q^{-1})\xi_k+\xi_{k-1}=\mathrm{const}$; hence $\xi_k$ is affine in the
$q$-quadratic node $q^{-k}+\tfrac ab q^{k-n}$. The resulting band entries determine the four
$q$-Racah parameters. The result is the following closed form, for which we first fix the
standing hypotheses.

\medskip
\noindent\textbf{Hypotheses (H).} $0<q<1$ and $a,b,c,d\in\Cx^{\times}$; the algebraic
statements of this section, in particular Theorem~\ref{thm:explicit}, require nothing
further than the nonvanishing conditions below. Reality assumptions enter only later and are
stated where used: the weights of Remark~\ref{rem:orth} are real when the parameters are
real; the integral statements of \S\ref{sec:wshift} assume $|a|,|b|,|c|,|d|<1$; and the
positivity discussion of Proposition~\ref{prop:obstruction}\textup{(iii)} assumes
$0<a,b<1$ with $q<a/b<q^{-1}$, $a\ne b$. The nonvanishing conditions are:
\begin{enumerate}[label=\textup{(H\arabic*)},leftmargin=2.6em,itemsep=1pt]
\item $a/b\notin\{q^{j}:\ -(n+1)\le j\le n+1\}$, \ $ab\notin\{q^{-j}:\ 0\le j\le n-1\}$, \
$a^{2},\,b^{2}\notin\{q^{-j}:\ 1\le j\le 2n-1\}$ \ \textup(the family \eqref{eq:genbern} is a
basis of $\Pi_n$; Lemma~\ref{lem:basis}\textup);
\item $ac,\,ad,\,bc,\,bd\notin\{q^{-j}:\ 0\le j\le n-1\}$ \ \textup(the band coefficients
$\mathfrak b_k,\mathfrak d_k$ of \eqref{eq:Ltridiag} do not vanish in the interior\textup);
\item $abcd\notin\{q^{\,1-j}:\ 1\le j\le 2n-1\}$ \ \textup(the eigenvalues
$\lambda_m=(q^{-m}-1)(1-abcd\,q^{m-1})$, $0\le m\le n$, are pairwise distinct, since
$\lambda_m-\lambda_{m'}=(q^{-m}-q^{-m'})(1-abcd\,q^{m+m'-1})$\textup).
\end{enumerate}
For the integral statements of \S\ref{sec:wshift} we additionally assume $|a|,|b|,|c|,|d|<1$,
so that the Askey--Wilson measure is the absolutely continuous one \eqref{eq:awweight}, without
discrete masses \cite[\S14.1]{KLS2010}. The following elementary lemma ensures that no hidden
nonvanishing conditions arise in the sequel.

\begin{lemma}[control of the denominators]\label{lem:denominators}
Assume \textup{(H1)}--\textup{(H3)}. Then, over the indicated ranges of indices, every
denominator occurring in \eqref{eq:prefactor}, \eqref{eq:dk}, \eqref{eq:bk}, \eqref{eq:rho},
\eqref{eq:Bq} and \eqref{eq:Dq} is nonzero, every numerator factor of \eqref{eq:rho} is
nonzero, and the quantities $B(k)$ \textup(for $0\le k\le n-1$\textup) and $D(k)$
\textup(for $1\le k\le n$\textup) are nonzero.
\end{lemma}
\begin{proof}
Each factor in question has one of four shapes: $1-q^{\,t}$ with $1\le t\le n$, never zero
since $0<q<1$; $1-\tfrac ab q^{\,t}$ or $1-\tfrac ba q^{\,t}$ with $|t|\le n+1$, excluded by
\textup{(H1)}; $1-abq^{\,t}$ with $0\le t\le n-1$, excluded by \textup{(H1)}; or
$1-xyq^{\,t}$ with $xy\in\{ac,ad,bc,bd\}$ and $0\le t\le n-1$, excluded by \textup{(H2)}.
The exponent ranges are read off directly: in \eqref{eq:dk} the denominators involve
$\tfrac ab q^{\,2k-n-1}$ and $\tfrac ab q^{\,2k-n}$, whose exponents run through
$[-n-1,\,n]$ as $0\le k\le n$; in \eqref{eq:bk} the exponents of $\tfrac ba$ run through
$[-n-1,\,n]$; in \eqref{eq:rho} through $[-n,\,n+1]$; and in \eqref{eq:Bq}--\eqref{eq:Dq}
the combinations $\gamma\delta q^{\,2k+i}=\tfrac ab q^{\,2k-n-1+i}$, $i=0,1,2$, have
exponents in $[-n-1,\,n+1]$. The numerator factors of $B(k)$ for $k<n$ and of $D(k)$ for
$k>0$ are $1-adq^{\,k}$, $1-acq^{\,k}$, $1-q^{\,k-n}$, $1-\gamma\delta q^{\,k+1}$,
$1-q^{\,k}$, $1-\tfrac ab q^{\,k}$, $\beta-\gamma q^{\,k}$ and
$\alpha-\gamma\delta q^{\,k}$; the first six fall into the families above, while
$\beta-\gamma q^{\,k}=\beta\big(1-\tfrac{q^{\,k-n}}{bc}\big)$ and
$\alpha-\gamma\delta q^{\,k}=\alpha\big(1-\tfrac{q^{\,k-n}}{bd}\big)$, both nonzero by
\textup{(H2)} since $0\le n-k\le n-1$ in the relevant range.
\end{proof}

{
\begin{theorem}[Askey--Wilson/$q$-Racah connection]\label{thm:explicit}
Assume \textup{(H)}. With the $q$-Racah polynomials \eqref{eq:qRacah}, set
\begin{equation}\label{eq:params}
\alpha=\frac{ad}{q},\qquad \beta=\frac{bc}{q},\qquad \gamma=q^{-n-1},\qquad \delta=\frac ab,
\end{equation}
so that $\alpha q=ad$, $\beta\delta q=ac$, $\gamma q=q^{-n}$ \textup(the standard truncation\textup),
$\alpha\beta q^{m+1}=abcd\,q^{m-1}$ and $\mu(k)=q^{-k}+\gamma\delta q^{k+1}=q^{-k}+\tfrac ab q^{k-n}$.
Then the connection coefficients \eqref{eq:conn} are, for $0\le k,m\le n$,
\begin{equation}\label{eq:explicit}
\begin{gathered}
A_k(n,m)=\pi_{n,k}\,\omega_{n,m}\,R_m\!\big(\mu(k);\alpha,\beta,\gamma,\delta\,|\,q\big),\\[2pt]
R_m\!\big(\mu(k);\alpha,\beta,\gamma,\delta\,|\,q\big)=\qhyp43\!\left[\genfrac{}{}{0pt}{}{q^{-m},\ abcd\,q^{m-1},\ q^{-k},\ \tfrac ab q^{k-n}}{ad,\ ac,\ q^{-n}};q,q\right],
\end{gathered}
\end{equation}
with the \emph{fully explicit} prefactors
\begin{align}
\pi_{n,k}&=\frac{q^{\,k}}{\dbinom nk}\,\qbinom nk\,
\frac{1-\tfrac ba q^{\,n-2k}}
{\ph{ab}{n}\,\ph{b/a}{n-k}\,\big(1-\tfrac ba q^{\,n-k}\big)\,\ph{\tfrac ab q}{k}},
\label{eq:prefactor}\\[2pt]
\omega_{n,m}&=a^{-m}\,\ph{ab}{m}\,\ph{ac}{m}\,\ph{ad}{m},\label{eq:omega}
\end{align}
where $\pi_{n,k}=A_k(n,0)$ expands unity, $1=\sum_k\pi_{n,k}\widetilde\Bs^{\,n}_k$
\textup(it depends only on $a,b,q$ and is the normalising constant of the Askey--Wilson Bernstein basis of \cite{SimeonovGoldman2019}, Remark~\ref{rem:SG}\textup), and $\omega_{n,m}=A_0(n,m)/A_0(n,0)$ is
independent of $n$. The apparent singularities in \eqref{eq:prefactor} at special parameter values are
interpreted through rational continuation of the complete connection coefficient
whenever the connection matrix remains well defined. In particular
\[
\pi_{n,0}=\frac{1}{\ph{b/a}{n}\,\ph{ab}{n}},\qquad
\pi_{n,n}=\frac{1}{\ph{a/b}{n}\,\ph{ab}{n}} .
\]
\end{theorem}
}

Throughout the proof we work with the unnormalised products
\[
\Phi_k:=\varphi_k(x;a)\,\varphi_{n-k}(x;b), 
\]
so that
$\widetilde\Bs^{\,n}_k=\binom nk\Phi_k$, and with the correspondingly rescaled coefficients
$\widehat A_k:=\binom nk A_k(n,m)$ and $\widehat\pi_k:=\binom nk\pi_{n,k}$; all statements
carry over unchanged to $\widetilde\Bs^{\,n}_k$. The proof is entirely self-contained: it
consists of a pole analysis in the Laurent variable $z$ (Step~1), elementary manipulations of
finite products (Steps~3, 4 and 6, the algebraic details of which are collected in Appendix~\ref{app:certificates}), and a piece of finite-dimensional spectral theory
(Step~5).

\begin{proof}
{
We first work on the Zariski-open set of parameter values for which the candidate poles used below are simple and pairwise distinct. Every identity obtained there is an identity between rational functions of $a,b,c,d,q^k$ and $q^n$. It therefore extends by rational continuation to all parameter values satisfying the nonvanishing hypotheses \textup{(H1)}--\textup{(H3)}.
}
Let $\mathcal L=\mathcal L_{a,b,c,d}$ be the second-order Askey--Wilson divided-difference
operator, for which $\mathcal L p_m=\lambda_m p_m$ with
$\lambda_m=(q^{-m}-1)(1-abcd\,q^{m-1})$ \cite[\S16.5]{Ismail2005}. Writing
$z=\eit$ and letting $\eta$ denote the $q$-shift $z\mapsto qz$,
$\mathcal L$ acts on symmetric Laurent polynomials $f=f\big(\tfrac12(z+z^{-1})\big)$ by
\begin{equation}\label{eq:Lform}
\mathcal L f=A(z)\,(\eta f-f)+A(z^{-1})\,(\eta^{-1}f-f),\qquad
A(z)=\frac{(1-az)(1-bz)(1-cz)(1-dz)}{(1-z^{2})(1-qz^{2})}.
\end{equation}

\emph{Step 1 (tridiagonality, with explicit band).} We claim
\begin{equation}\label{eq:Ltridiag}
\mathcal L\,\Phi_k=\mathfrak b_{k}\,\Phi_{k+1}
+\mathfrak a_{k}\,\Phi_{k}+\mathfrak d_{k}\,\Phi_{k-1}
\qquad(0\le k\le n,\ \Phi_{-1}=\Phi_{n+1}:=0),
\end{equation}
with
\begin{align}
\mathfrak d_k&=\frac ab\,q^{\,2k-2n-1}\,
\frac{(1-q^{k})\big(1-\tfrac ab q^{k}\big)(1-acq^{k-1})(1-adq^{k-1})}
{\big(1-\tfrac ab q^{\,2k-n-1}\big)\big(1-\tfrac ab q^{\,2k-n}\big)},\label{eq:dk}\\
\mathfrak b_k&=\frac ba\,q^{-2k-1}\,
\frac{(1-q^{\,n-k})\big(1-\tfrac ba q^{\,n-k}\big)(1-bcq^{\,n-k-1})(1-bdq^{\,n-k-1})}
{\big(1-\tfrac ba q^{\,n-2k}\big)\big(1-\tfrac ba q^{\,n-2k-1}\big)},\label{eq:bk}
\end{align}
so that in particular $\mathfrak b_k=\mathfrak d_{n-k}\big|_{a\leftrightarrow b}$ and
$\mathfrak d_0=\mathfrak b_n=0$, in agreement with the boundary convention; the diagonal
entry $\mathfrak a_k$ is defined in the course of the argument. Divide the desired identity
by $\Phi_k$ and set
\[
H(z):=\frac{\mathcal L\Phi_k(z)}{\Phi_k(z)}=G(z)+G(1/z),\qquad
G(z):=A(z)\big(R(z)-1\big),\qquad R(z):=\frac{\Phi_k(qz)}{\Phi_k(z)},
\]
where the symmetry $\Phi_k(z)=\Phi_k(1/z)$ has been used to write
$\Phi_k(z/q)/\Phi_k(z)=R(1/z)$. Since $\varphi_k(x;a)=(az;q)_k(a/z;q)_k$, the shifted
quotient is the explicit rational function
\begin{equation}\label{eq:Rquot}
R(z)=\frac{(1-aq^{k}z)\big(1-\tfrac aq z^{-1}\big)(1-bq^{\,n-k}z)\big(1-\tfrac bq z^{-1}\big)}
{(1-az)\,\big(1-aq^{k-1}z^{-1}\big)\,(1-bz)\,\big(1-bq^{\,n-k-1}z^{-1}\big)},
\end{equation}
each of the four telescopings being of the form $(cq z;q)_k/(cz;q)_k=(1-cq^{k}z)/(1-cz)$.
Likewise
\begin{gather}
T_+(z):=\frac{\Phi_{k+1}(z)}{\Phi_k(z)}
=\frac{(1-aq^{k}z)(1-aq^{k}/z)}{(1-bq^{\,n-k-1}z)(1-bq^{\,n-k-1}/z)},\notag\\
T_-(z):=\frac{\Phi_{k-1}(z)}{\Phi_k(z)}
=\frac{(1-bq^{\,n-k}z)(1-bq^{\,n-k}/z)}{(1-aq^{k-1}z)(1-aq^{k-1}/z)}.\label{eq:Tquot}
\end{gather}
We now locate the poles of $H$. The candidate poles of $G(z)$ are the zeros of the
denominators in \eqref{eq:Lform} and \eqref{eq:Rquot}, namely
$z\in\{\pm1,\pm q^{-1/2}\}\cup\{1/a,\,1/b\}\cup\{aq^{k-1},\,bq^{\,n-k-1}\}$; on the generic parameter set fixed at the beginning of the proof these points are simple and pairwise distinct. The resulting rational identity extends to all parameters covered by \textup{(H)}. Four cancellation mechanisms dispose of all but the last pair.
\begin{enumerate}[label=\textup{(\alph*)},leftmargin=2.2em,itemsep=1pt]
\item At $z=1/a$ and $z=1/b$: the numerator of $A(z)$ contains the factors $(1-az)(1-bz)$,
which cancel the corresponding denominators of $R(z)$; hence $G$ is regular there.
\item At $z=\pm q^{-1/2}$, that is at $qz^{2}=1$: here $qz=1/z$, so
$\Phi_k(qz)=\Phi_k(1/z)=\Phi_k(z)$ by symmetry, whence $R(z)=1$; the simple zero of $R-1$
cancels the simple pole of $A$. (This is precisely the mechanism by which $\mathcal L$
preserves symmetric polynomials.)
\item At $z=\pm1$: the function $H(z)=G(z)+G(1/z)$ is invariant under $z\mapsto1/z$, a map
which fixes $\pm1$; if $G(z)\sim C/(1-z)$ near $z=1$ then $G(1/z)\sim -C/(1-z)$ there, so the
two simple poles cancel, and similarly at $z=-1$.
\item The poles of $G(1/z)$ are treated symmetrically and contribute the reciprocal points.
\end{enumerate}
Consequently $H$ is a rational function whose only poles, all simple, lie at
$z^{\pm1}=aq^{k-1}$ and $z^{\pm1}=bq^{\,n-k-1}$; these are exactly the poles of $T_-$ and
$T_+$ respectively. Define
\begin{gather*}
\mathfrak d_k:=\frac{\operatorname*{Res}_{z=aq^{k-1}}H(z)}{\operatorname*{Res}_{z=aq^{k-1}}T_-(z)},\qquad
\mathfrak b_k:=\frac{\operatorname*{Res}_{z=bq^{\,n-k-1}}H(z)}{\operatorname*{Res}_{z=bq^{\,n-k-1}}T_+(z)},\\
\mathfrak a_k:=\lim_{z\to\infty}\big(H-\mathfrak b_kT_+-\mathfrak d_kT_-\big)(z);
\end{gather*}
the residue quotients are evaluated in Appendix~\ref{app:certificates}
(Proposition~\ref{prop:residues}) and give precisely \eqref{eq:dk} and \eqref{eq:bk}, while
the limit defining $\mathfrak a_k$ exists because $H$, $T_+$ and $T_-$ all possess finite
limits at $z\to\infty$ (displayed in \eqref{eq:infvalues} below). The difference
\[
\Delta(z):=H(z)-\mathfrak b_k\,T_+(z)-\mathfrak a_k-\mathfrak d_k\,T_-(z)
\]
is symmetric under $z\mapsto1/z$, has no poles (the residues at $z^{\pm1}=aq^{k-1}$ and
$z^{\pm1}=bq^{\,n-k-1}$ vanish by construction, once at each point and hence, by symmetry, at
its reciprocal), and is bounded as $z\to\infty$; a rational function without poles on the
Riemann sphere is constant, and the constant is $\Delta(\infty)=0$ by the choice of
$\mathfrak a_k$. Multiplying through by $\Phi_k$ yields \eqref{eq:Ltridiag}. For later use we
record the values at infinity,
\begin{equation}\label{eq:infvalues}
H(\infty)=\frac{abcd}{q}\,(q^{n}-1)+(q^{-n}-1),\qquad
T_+(\infty)=\frac ab\,q^{\,2k-n+1},\qquad
T_-(\infty)=\frac ba\,q^{\,n-2k+1}.
\end{equation}

\emph{Step 2 (difference equation for the coefficients).} Insert
$p_m=\sum_k\widehat A_k\Phi_k$ into $\mathcal Lp_m=\lambda_mp_m$ and equate coefficients of
$\Phi_k$:
\begin{equation}\label{eq:qRacahdiff}
\mathfrak d_{k+1}\widehat A_{k+1}(n,m)+\mathfrak a_{k}\widehat A_{k}(n,m)
+\mathfrak b_{k-1}\widehat A_{k-1}(n,m)=\lambda_m\,\widehat A_k(n,m).
\end{equation}

\emph{Step 3 (gauge).} Introduce the explicit rational quantity
\begin{equation}\label{eq:rho}
\rho_k:=q^{-1}\,
\frac{(1-q^{\,n-k})\big(1-\tfrac ba q^{\,n-k}\big)\big(1-\tfrac ab q^{\,2k+2-n}\big)}
{(1-q^{\,k+1})\big(1-\tfrac ab q^{\,2k-n}\big)\big(1-\tfrac ab q^{\,k+1}\big)},
\qquad 0\le k\le n-1,
\end{equation}
whose numerators and denominators are nonzero under \textup{(H)}
(Lemma~\ref{lem:denominators}), and the gauge sequence
$\widehat c_0:=1$, $\widehat c_{k+1}:=\rho_k\,\widehat c_k$. A direct multiplication of
finite products, proved algebraically in Appendix~\ref{app:certificates}
(Proposition~\ref{prop:BD}), shows that
\begin{equation}\label{eq:BDdef}
B(k):=\mathfrak d_{k+1}\,\rho_k,\qquad D(k):=\frac{\mathfrak b_{k-1}}{\rho_{k-1}}
\end{equation}
simplify to
\begin{align}
B(k)&=\frac{(1-\alpha q^{k+1})(1-\beta\delta q^{k+1})(1-\gamma q^{k+1})(1-\gamma\delta q^{k+1})}
{(1-\gamma\delta q^{2k+1})(1-\gamma\delta q^{2k+2})},\label{eq:Bq}\\
D(k)&=q\,\frac{(1-q^{k})(1-\delta q^{k})(\beta-\gamma q^{k})(\alpha-\gamma\delta q^{k})}
{(1-\gamma\delta q^{2k})(1-\gamma\delta q^{2k+1})},\label{eq:Dq}
\end{align}
with $(\alpha,\beta,\gamma,\delta)$ as in \eqref{eq:params}: these are exactly the
coefficients of the standard $q$-Racah second-order difference equation
\cite[Eq.~(14.2.6)]{KLS2010}. Observe that $D(0)=0$ (factor $1-q^{k}$) and $B(n)=0$ (factor
$1-\gamma q^{k+1}=1-q^{\,k-n}$).

\emph{Step 4 (the diagonal).} The diagonal entry defined in Step~1 satisfies
\begin{equation}\label{eq:diagid}
\mathfrak a_k=-\big(B(k)+D(k)\big),\qquad 0\le k\le n .
\end{equation}
This is an identity between two explicitly displayed rational functions of $q^{k}$; it is
established in Appendix~\ref{app:certificates} (Proposition~\ref{prop:diag}) by verifying
that the difference of the two sides, a rational function of $q^{k}$ with at most six simple
poles, has vanishing residue at each of them and tends to zero as $q^{k}\to\infty$.
Consequently, in the gauge $\widehat A_k(n,m)=\widehat c_k\,y_k$ the recurrence
\eqref{eq:qRacahdiff} becomes
\begin{equation}\label{eq:gauged}
B(k)\,y_{k+1}-\big(B(k)+D(k)\big)y_k+D(k)\,y_{k-1}=\lambda_m\,y_k,\qquad 0\le k\le n,
\end{equation}
which coincides exactly, after the parameter identification \eqref{eq:params}, with the
standard $q$-Racah difference equation, whose eigenvalue
$(q^{-m}-1)(1-\alpha\beta q^{m+1})$ equals $\lambda_m$ because
$\alpha\beta q^{m+1}=abcd\,q^{m-1}$.

\emph{Step 5 (spectral resolution).} Let $T$ denote the $(n+1)\times(n+1)$ matrix of the
gauged operator on the left-hand side of \eqref{eq:gauged}. For each $0\le j\le n$ the vector
\[
\big(R_j(\mu(k);\alpha,\beta,\gamma,\delta\,|\,q)\big)_{k=0}^{n}
\]
is an eigenvector of $T$
with eigenvalue $\lambda_j$: in the interior this is the $q$-Racah difference equation
\cite[Eq.~(14.2.6)]{KLS2010}, and at the boundary rows $k=0$ and $k=n$ the equation persists
because $D(0)=0$ and $B(n)=0$. By hypothesis \textup{(H3)} the numbers
$\lambda_0,\dots,\lambda_n$ are pairwise distinct, so these $n+1$ eigenvectors are linearly
independent, exhaust the spectrum of $T$, and every eigenspace is one-dimensional. The
vector $\big(y_k\big)_k=\big(\widehat A_k(n,m)/\widehat c_k\big)_k$ satisfies
$Ty=\lambda_m y$ by Step~4, whence
\[
\widehat A_k(n,m)=c_m\,\widehat c_k\,R_m\big(\mu(k);\alpha,\beta,\gamma,\delta\,|\,q\big)
\]
for some scalar $c_m\ne0$.

\emph{Step 6 (prefactors and normalisation).} It remains to identify $\widehat c_k$ and
$c_m$. Telescoping the product $\widehat c_k=\prod_{j=0}^{k-1}\rho_j$ factor by factor,
\[
\prod_{j=0}^{k-1}\frac{1-q^{\,n-j}}{1-q^{\,j+1}}=\qbinom nk,\qquad
\prod_{j=0}^{k-1}\big(1-\tfrac ba q^{\,n-j}\big)=\ph{\tfrac ba q^{\,n-k+1}}{k},
\]
\[
\prod_{j=0}^{k-1}\frac{1-\tfrac ab q^{\,2j+2-n}}{1-\tfrac ab q^{\,2j-n}}
=\frac{1-\tfrac ab q^{\,2k-n}}{1-\tfrac ab q^{-n}},
\]
together with $\prod_{j=0}^{k-1}q^{-1}=q^{-k}$ and
$\prod_{j=0}^{k-1}\big(1-\tfrac ab q^{\,j+1}\big)^{-1}=\ph{\tfrac ab q}{k}^{-1}$, and
converting the two factors with negative exponents by means of
$1-\tfrac ab q^{-t}=-\tfrac ab q^{-t}\big(1-\tfrac ba q^{t}\big)$, one obtains from the finite-product calculation in Appendix~\ref{app:certificates}, Proposition~\ref{prop:telescope}
\begin{equation}\label{eq:chat}
\widehat c_k=\ph{ab}{n}\,\ph{b/a}{n}\;\times\;
q^{\,k}\,\qbinom nk\,
\frac{1-\tfrac ba q^{\,n-2k}}
{\ph{ab}{n}\,\ph{b/a}{n-k}\,\big(1-\tfrac ba q^{\,n-k}\big)\,\ph{\tfrac ab q}{k}} .
\end{equation}
For $m=0$ one has $p_0=1$ and $R_0\equiv1$, so $\widehat A_k(n,0)=c_0\widehat c_k$;
evaluating the expansion $1=\sum_k\widehat A_k(n,0)\Phi_k$ at the point
$x_0(a)=\tfrac12(a+a^{-1})$, where every $\Phi_k$ with $k\ge1$ vanishes
(Proposition~\ref{prop:obstruction}), gives
$\widehat A_0(n,0)=1/\big[\ph{ab}{n}\ph{b/a}{n}\big]$, whence
$c_0=1/\big[\ph{ab}{n}\ph{b/a}{n}\big]$ and $\widehat\pi_k=c_0\widehat c_k$ is precisely the
closed form \eqref{eq:prefactor} after reinstating the binomials. For general $m$, evaluate
$p_m=\sum_k\widehat A_k\Phi_k$ at $x_0(a)$: only $k=0$ survives, $R_m(\mu(0))=1$ because
$\ph{q^{0}}{j}=0$ for $j\ge1$, and the ${}_4\phi_3$ defining $p_m$ in \eqref{eq:AW}
collapses to $1$ at $z=a$ (the entry $\ph{a/z}{j}$ vanishes for $j\ge1$), so that
$p_m(x_0(a))=a^{-m}\ph{ab,ac,ad}{m}$. Dividing by
$\Phi_0(x_0(a))=\ph{ab}{n}\ph{b/a}{n}$ yields
$c_m/c_0=a^{-m}\ph{ab,ac,ad}{m}=\omega_{n,m}$, which is \eqref{eq:omega}. Assembling,
$A_k(n,m)=\pi_{n,k}\,\omega_{n,m}\,R_m(\mu(k))$, which is \eqref{eq:explicit}.
\end{proof}

\begin{remark}[standard truncation]\label{rem:notqracah}
The truncation is the standard one, $\gamma q=q^{-n}$: the family
\eqref{eq:explicit} is a genuine $q$-Racah polynomial of degree $m$ on the finite lattice
$k=0,1,\dots,n$, not an analytic continuation. Two of the three lower ${}_4\phi_3$
parameters, $ad$ and $ac$, together with the numerator entry $abcd\,q^{m-1}$, are data of
the target Askey--Wilson polynomial \eqref{eq:AW} (the third lower parameter, $q^{-n}$, is
the degree truncation), which is the structural reason the identification is so rigid.
\end{remark}

\begin{remark}[quasi-definite orthogonality in the B\'ezier index]\label{rem:orth}
Under \textup{(H2)} the matrix $T$ of \eqref{eq:gauged}, with entries
$T_{k,k+1}=B(k)$, $T_{k,k}=-\big(B(k)+D(k)\big)$ and $T_{k,k-1}=D(k)$, is irreducible
tridiagonal, and a diagonal symmetrisation is available in closed form: setting $h_0:=1$ and
\begin{equation}\label{eq:symm}
h_{k+1}:=h_k\,\frac{B(k)}{D(k+1)},\qquad 0\le k\le n-1,
\end{equation}
which is well defined and nonzero because $B(k)\ne0$ for $0\le k<n$ and $D(k)\ne0$ for
$0<k\le n$ (Lemma~\ref{lem:denominators}), one obtains
\[
h_kB(k)=h_{k+1}D(k+1),\qquad 0\le k\le n-1,
\]
which is precisely the entrywise condition, in position $(k,k+1)$, for
\[
HT=T^{\mathsf T}H,\qquad H:=\operatorname{diag}(h_0,\dots,h_n).
\]
Hence $T$ is symmetric with respect to the nondegenerate bilinear form
\[
\langle u,v\rangle_h=\sum_k h_ku_kv_k, 
\]
and, the eigenvalues being pairwise distinct under
\textup{(H3)}, eigenvectors associated with different eigenvalues are orthogonal for this
form; the functional it defines on polynomials in $\mu(k)$ is quasi-definite up to degree
$n$, and when the parameters are real the weights $h_k$ are real, though in general of mixed
sign. Consequently
\begin{equation}\label{eq:coeffOrth}
\sum_{k=0}^{n}\rho_A(k)\,A_k(n,m)\,A_k(n,m')=0\quad(m\ne m'),\qquad
\rho_A(k):=\frac{h_k}{\pi_{n,k}^{2}} ,
\end{equation}
where $h_k$ coincides, up to a nonzero common factor, with the standard $q$-Racah weight for
the parameters \eqref{eq:params}, whose consecutive ratio is exactly $B(k)/D(k+1)$: the
weight for the coefficients themselves absorbs the factor $\pi_{n,k}^{-2}$, and it is $h_k$
alone only for the gauged ratios $A_k(n,m)/(\pi_{n,k}\omega_{n,m})$. The classification of
the parameter regions in which $\rho_A$, or $h$, is positive is an interesting problem which
we do not pursue here.
\end{remark}

\begin{remark}[possible alternative approaches]\label{rem:engine}
Two further routes to \eqref{eq:explicit} suggest themselves, although we do not carry
either out here: a moment route, expanding $p_m(\cdot;a,b,c,d)$ in the shifted family
$p_l(\cdot;aq^{k},bq^{n-k},c,d)$ and integrating against the shifted weight of
Lemma~\ref{lem:wshift}, so that only $l=0$ survives and $A_k(n,m)$ is expressed through an
Askey--Wilson connection coefficient and the Askey--Wilson integral \cite[\S15.2]{Ismail2005};
and a structural route, since the pair formed by $\mathcal L$ and the multiplication
operator by $\mu(k)$, acting on the finite module spanned by
$\{\widetilde\Bs^{\,n}_k\}$, suggests a possible Leonard-pair or tridiagonal-pair
interpretation of \eqref{eq:Ltridiag}; establishing that the axioms of a Leonard pair
(diagonalisability of both operators, and irreducible tridiagonality of each in an eigenbasis
of the other) are actually met would give a representation-theoretic explanation of the
present results, and we leave it as a structural conjecture.
\end{remark}

\section{Limit transitions: recovering the known results}\label{sec:limits}

The vertical arrows of the Askey scheme are realised as limits of the lattice
\eqref{eq:latticetypes}; the individual arrows, each with its own scaling, are documented in
\cite[\S14]{KLS2010} and \cite[\S18.28]{DLMF}. It must be stressed that every arrow requires
an \emph{explicit parametrised scaling}: fixing $a,b,c,d,x$ and letting $q\to1$ does
\emph{not} produce the quadratic-lattice objects. Because all sums in \eqref{eq:conn} are
finite, once the scaling is fixed no analytic interchange of limit and summation is required
and the reductions are term-by-term algebraic limits.

\begin{center}
\renewcommand{\arraystretch}{1.3}
\begin{tabular}{@{}p{3.9cm}p{4.2cm}p{2.9cm}p{3.0cm}@{}}
\hline
\textbf{Lattice / limit} & \textbf{Scaling} & \textbf{Bernstein / target} & \textbf{Coefficients}\\
\hline
$q$-quadratic (this paper) & none & gen.-power / Askey--Wilson & $q$-Racah
(Thm.~\ref{thm:explicit})\\
$q$-quadratic $\xrightarrow{q\uparrow1}$ quadratic & $a{=}q^{A},b{=}q^{B},c{=}q^{C},d{=}q^{D}$
& Wilson powers / Wilson & Racah (\S\ref{sec:qtoone})\\
$q$-quadratic $\to$ $q$-linear & $b=\Lambda\to\infty$, $c=\widehat c/\Lambda$ &
normalised coefficients & $q$-Hahn (\S\ref{sec:qlinear})\\
$q$-linear $\xrightarrow{q\uparrow1}$ linear & $x$ fixed, $q\to1$ & discrete Bernstein / Hahn &
Hahn--Eberlein\ \cite{RAGZ1998}\\
\hline
\end{tabular}
\end{center}

\subsection{The limit $q\uparrow1$: from the Askey--Wilson to the Wilson connection}
\label{sec:qtoone}
Two scalings are required, one for the parameters and one for the variable. Substitute
\begin{equation}\label{eq:scaling}
a=q^{A},\qquad b=q^{B},\qquad c=q^{C},\qquad d=q^{D},\qquad z=q^{\,iX},
\end{equation}
so that $x=\cos\theta=\tfrac12(q^{\,iX}+q^{-iX})$; the $q$-Racah parameters
\eqref{eq:params} become $\alpha=q^{A+D-1}$, $\beta=q^{B+C-1}$, $\gamma=q^{-n-1}$,
$\delta=q^{A-B}$. Under \eqref{eq:scaling} the Askey--Wilson monomials contract to the
Wilson generalized powers: each factor of $\varphi_k$ satisfies
$\big(1-q^{A+iX+j}\big)\big(1-q^{A-iX+j}\big)\big/(1-q)^{2}\to(A+iX+j)(A-iX+j)$, whence
\begin{equation}\label{eq:philimit}
\lim_{q\uparrow1}\frac{\varphi_k\big(x;q^{A}\big)}{(1-q)^{2k}}
=(A+iX)_k\,(A-iX)_k=\rho_k(X;A),
\end{equation}
the Wilson generalized power in the variable $X$ (a polynomial of degree $k$ in $X^{2}$).
Likewise, from the normalisation in \eqref{eq:AW},
\begin{equation}\label{eq:plimit}
\lim_{q\uparrow1}\frac{p_m\big(x;q^{A},q^{B},q^{C},q^{D}\,\big|\,q\big)}{(1-q)^{3m}}
=W_m\big(X^{2};A,B,C,D\big),
\end{equation}
the Wilson polynomial \cite[\S9.1, \S14.1]{KLS2010}. For the coefficients, every quotient in
the ${}_4\phi_3$ of \eqref{eq:explicit} is a ratio of $q$-shifted factorials with $q$-power
arguments, for which the limit is elementary and term-by-term,
$\ph{q^{Y}}{j}/(1-q)^{j}\to(Y)_j$; since the sum has $m+1$ terms and the powers of
$(1-q)$ cancel within each term, one obtains
\begin{multline*}
\lim_{q\uparrow1}
\qhyp43\!\left[\genfrac{}{}{0pt}{}{q^{-m},\,q^{\,A+B+C+D+m-1},\,q^{-k},\,q^{\,A-B+k-n}}
{q^{\,A+D},\ q^{\,A+C},\ q^{-n}};q,q\right]\\
=\hyp43\!\left[\genfrac{}{}{0pt}{}{-m,\,A{+}B{+}C{+}D{+}m{-}1,\,-k,\,A{-}B{+}k{-}n}
{A{+}D,\ A{+}C,\ -n};1\right]
=:R^{\mathrm{Racah}}_m\big(\lambda(k)\big),
\end{multline*}
the Racah polynomial with parameters $(A{+}D{-}1,\,B{+}C{-}1,\,-n{-}1,\,A{-}B)$ and
$\lambda(k)=k(k+A-B-n)$ \cite[\S9.2]{KLS2010}. The prefactor limits are equally explicit:
counting one power of $(1-q)$ for each linear factor, with numerators and denominators
counted with opposite signs, the exact normalisations are
\begin{align}
\bar\pi_{n,k}:=\lim_{q\uparrow1}\,(1-q)^{2n}\,\pi_{n,k}
&=\frac{B-A+n-2k}{(A{+}B)_n\,(B{-}A)_{n-k}\,(B{-}A{+}n{-}k)\,(A{-}B{+}1)_k},
\label{eq:pibar}\\
\bar\omega_{n,m}:=\lim_{q\uparrow1}\,(1-q)^{-3m}\,\omega_{n,m}
&=(A{+}B)_m\,(A{+}C)_m\,(A{+}D)_m ,\label{eq:omegabar}
\end{align}
the exponent $2n$ in \eqref{eq:pibar} arising as $n+(n-k)+1+k-1$ from the five factor
groups of \eqref{eq:prefactor}, while $q^{k}\to1$ and
$\qbinom nk\big/\binom nk\to1$. Here and in \eqref{eq:pibar} the exponents
$A,B,C,D$ are assumed generic, so that no factor such as $B-A+n-k$ or $B-A+n-2k$
vanishes. At exceptional parameter values the separately displayed factors may possess
removable singularities; the limiting connection identity is then understood by rational
continuation of the complete coefficient, not necessarily of each factor in isolation.

\begin{proposition}[the Wilson connection]\label{prop:racahlimit}
Under the scalings \eqref{eq:scaling}, divide the identity
$p_m=\sum_kA_k(n,m)\widetilde\Bs^{\,n}_k$ by $(1-q)^{3m}$ and insert the normalisations
\eqref{eq:philimit}, \eqref{eq:pibar} and \eqref{eq:omegabar}; every factor converges, the
powers of $(1-q)$ balance (the coefficient $A_k(n,m)$ carries the exact weight
$(1-q)^{3m-2n}$, while the two generalized powers carry $(1-q)^{2k}$ and
$(1-q)^{2(n-k)}$), and the limit is the Wilson connection identity
\[
W_m\big(X^{2};A,B,C,D\big)
=\sum_{k=0}^{n}\bar\pi_{n,k}\,\bar\omega_{n,m}\,
R^{\mathrm{Racah}}_m\big(\lambda(k)\big)\,
\binom nk\,\rho_k(X;A)\,\rho_{n-k}(X;B),
\]
each $\rho_k(X;A)=(A+iX)_k(A-iX)_k$ being a polynomial of degree $k$ in $X^{2}$, which
expresses the connection coefficients between the Wilson generalized powers and the
Wilson polynomials through Racah polynomials with parameters
\[
\big(A+D-1,\ B+C-1,\ -n-1,\ A-B\big).
\]
\end{proposition}

\begin{proof}
All sums are finite, so it suffices to verify the convergence of each factor and the balance
of the powers of $(1-q)$. The former is the content of \eqref{eq:philimit},
\eqref{eq:plimit}, \eqref{eq:pibar} and \eqref{eq:omegabar}, together with the displayed
limit of the ${}_4\phi_3$; the latter is the count just given, and the identification of the
limit ${}_4F_3$ as the Racah polynomial is \cite[\S9.2]{KLS2010}.
\end{proof}

{
\subsection{Degeneration to the $q$-linear level}\label{sec:qlinear}
We derive the $q$-linear degeneration directly from the closed formula of Theorem~\ref{thm:explicit}. Keep $q$, $a$, $d$, and $\widehat c$ fixed, let $\Lambda\to\infty$, and set
\begin{equation}\label{eq:qlinscaling}
b=\Lambda,
\qquad
c=\frac{\widehat c}{\Lambda}.
\end{equation}
Then
\[
\alpha=\frac{ad}{q},
\qquad
\beta=\frac{bc}{q}=\frac{\widehat c}{q},
\qquad
\gamma=q^{-n-1},
\qquad
\delta=\frac ab=\frac{a}{\Lambda}\longrightarrow0.
\]
In particular,
\[
\gamma\delta q^{k+1}=\frac{a}{\Lambda}q^{k-n}\longrightarrow0,
\qquad
\beta\delta q=ac=\frac{a\widehat c}{\Lambda}\longrightarrow0,
\]
and the $q$-quadratic node satisfies
\[
\mu_{\Lambda}(k)=q^{-k}+\frac{a}{\Lambda}q^{k-n}\longrightarrow q^{-k}.
\]
Since the basic hypergeometric series terminates, its limit may be taken term by term:
\begin{align}
&R_m\!\left(
\mu_{\Lambda}(k);
\frac{ad}{q},
\frac{\widehat c}{q},
q^{-n-1},
\frac{a}{\Lambda}
\,\middle|\,q\right)
\notag\\
&\quad={}_4\phi_3\!\left[
\genfrac{}{}{0pt}{}
{q^{-m},\,a\widehat c d\,q^{m-1},\,q^{-k},\,\frac{a}{\Lambda}q^{k-n}}
{ad,\,\frac{a\widehat c}{\Lambda},\,q^{-n}};q,q\right]
\notag\\
&\quad\longrightarrow
{}_3\phi_2\!\left[
\genfrac{}{}{0pt}{}
{q^{-m},\,a\widehat c d\,q^{m-1},\,q^{-k}}
{ad,\,q^{-n}};q,q\right].
\label{eq:qhahnlimit}
\end{align}
With the normalisation
\[
Q_m(q^{-k};\alpha',\beta',n\mid q)
:={}_3\phi_2\!\left[
\genfrac{}{}{0pt}{}
{q^{-m},\,\alpha'\beta' q^{m+1},\,q^{-k}}
{\alpha' q,\,q^{-n}};q,q\right],
\]
the limiting polynomial equals
\begin{equation}\label{eq:qhahn-identification}
Q_m\!\left(q^{-k};\frac{ad}{q},\frac{\widehat c}{q},n\,\middle|\,q\right),
\end{equation}
because $\alpha' q=ad$ and
\[
\alpha'\beta' q^{m+1}
=\frac{ad}{q}\frac{\widehat c}{q}q^{m+1}
=a\widehat c d\,q^{m-1}.
\]

The prefactors have explicit limits. From~\eqref{eq:omega} and
\[
(\Lambda u;q)_r\sim(-\Lambda u)^r q^{\binom r2},
\qquad \Lambda\to\infty,
\]
we obtain
\begin{equation}\label{eq:omega-qlinear-limit}
\lim_{\Lambda\to\infty}\Lambda^{-m}\omega_{n,m}
=(-1)^m q^{\binom m2}(ad;q)_m.
\end{equation}
Likewise, using~\eqref{eq:prefactor},
\begin{equation}\label{eq:pi-qlinear-limit}
\lim_{\Lambda\to\infty}\Lambda^{2n-k}\pi_{n,k}
=(-1)^k a^{-k}
q^{-\binom n2-\binom{n-k}{2}}
\frac{\qbinom nk}{\binom nk}.
\end{equation}
Indeed, the quotient of the two linear factors involving $b/a$ tends to $q^{-k}$, which cancels the leading factor $q^k$, while
\[
(a\Lambda;q)_n\sim(-a\Lambda)^n q^{\binom n2},
\qquad
(\Lambda/a;q)_{n-k}
\sim\left(-\frac{\Lambda}{a}\right)^{n-k}q^{\binom{n-k}{2}},
\]
and $((a/\Lambda)q;q)_k\to1$.

\begin{proposition}[the $q$-Hahn degeneration]\label{prop:limits}
Under the scaling~\eqref{eq:qlinscaling}, for every $0\le k,m\le n$,
\[
R_m\!\left(
q^{-k}+\frac{a}{\Lambda}q^{k-n};
\frac{ad}{q},
\frac{\widehat c}{q},
q^{-n-1},
\frac{a}{\Lambda}
\,\middle|\,q\right)
\longrightarrow
Q_m\!\left(q^{-k};\frac{ad}{q},\frac{\widehat c}{q},n\,\middle|\,q\right).
\]
Moreover, the prefactors satisfy~\eqref{eq:omega-qlinear-limit} and~\eqref{eq:pi-qlinear-limit}. Thus the polynomial part of the suitably normalized connection coefficients degenerates from $q$-Racah to $q$-Hahn on the $q$-linear lattice $q^{-k}$.
\end{proposition}

\begin{proof}
All the series are terminating. Under~\eqref{eq:qlinscaling}, the upper parameter $(a/\Lambda)q^{k-n}$ and the lower parameter $a\widehat c/\Lambda$ tend to zero, so the associated finite $q$-shifted factorials tend to one. This proves~\eqref{eq:qhahnlimit}. The parameter identification~\eqref{eq:qhahn-identification} follows directly from the defining ${}_3\phi_2$ series. The prefactor limits follow from the displayed finite-product asymptotics.
\end{proof}

\begin{remark}\label{rem:qlinearbasis}
Proposition~\ref{prop:limits} establishes the $q$-Hahn degeneration of the normalised
connection coefficients. It does not, by itself, identify the unscaled functions
$\widetilde\Bs_k^n(x;a,\Lambda)$ with a fixed $q$-Bernstein basis, such as those of
\cite{Phillips1997,OrucPhillips2003} and their shape-parameter extensions \cite{QinHu2013}.
Such an identification also requires a $k$-dependent normalisation of
$\varphi_{n-k}(x;\Lambda)$ and the affine degeneration of the polynomial variable that
converts the $q$-quadratic lattice into a $q$-linear lattice; we leave it as an open task.
\end{remark}
}

\section{A numerical case study: airfoil least squares in the Askey--Wilson basis}\label{sec:apps}

This section is a numerical experiment, deliberately separated from the proofs in
\S\S\ref{sec:obstruction}--\ref{sec:limits}. Its relation to the theoretical
problem is precise but two-stage. A degree-$n$ approximation of the CST shape function is an
element $P_n\in\Pi_n$ of the same polynomial space in which the connection problem
\eqref{eq:conn} is posed. The polynomial is first fitted and evaluated in the Askey--Wilson
basis; Theorem~\ref{thm:explicit} then converts its Askey--Wilson coefficient vector, without
changing $P_n$, into coefficients of the generalized-power basis \eqref{eq:genbern}. CST and
the least-squares error model are external numerical ingredients, not consequences of the
connection theorem. The experiment also compares the affine Bernstein basis
\eqref{eq:bern}, representing the geometric side of the dichotomy, with the Askey--Wilson
representation, representing its orthogonal side. Full reproducibility data are collected in
Table~\ref{tab:repro}.

\subsection{The test geometry}
As data we take the NACA~2412 aerofoil, the section of the Cessna~172 wing: maximum camber
$2\%$ at $40\%$ chord and maximum thickness $12\%$ at $30\%$ chord. To remove any ambiguity
between the open and closed trailing-edge variants of the four-digit family, we record the
formulas actually used, with $x\in[0,1]$ the chord fraction and $t=0.12$, $\mathfrak m=0.02$,
$\mathfrak p=0.4$: the closed-trailing-edge thickness distribution
\[
y_t(x)=5t\big(0.2969\sqrt x-0.1260\,x-0.3516\,x^{2}+0.2843\,x^{3}-0.1036\,x^{4}\big),
\]
the piecewise parabolic camber line
\[
y_c(x)=
\begin{cases}
\dfrac{\mathfrak m}{\mathfrak p^{2}}\big(2\mathfrak p\,x-x^{2}\big), & 0\le x\le\mathfrak p,\\[6pt]
\dfrac{\mathfrak m}{(1-\mathfrak p)^{2}}\big((1-2\mathfrak p)+2\mathfrak p\,x-x^{2}\big),
& \mathfrak p\le x\le1,
\end{cases}
\]
and the surfaces $x_{u,l}=x\mp y_t\sin\theta_c$, $y_{u,l}=y_c\pm y_t\cos\theta_c$ with
$\theta_c=\arctan y_c'$. All fits below are performed on the ordinates $y_{u}$ and $y_{l}$
regarded as functions of the standard NACA construction parameter $x$, not of the displaced
physical abscissae $x_{u}$ and $x_{l}$; accordingly, the reported errors are ordinate errors
at equal values of the construction parameter. Thus the reported discrepancy is a
parameter-matched ordinate error, not a Euclidean or normal-distance error between the two
physical parametric curves. The ordinates exhibit the classical leading-edge square-root
behaviour $y\sim\sqrt{x}$ as $x\to0$ (diverging slope; nose radius
$r_{\mathrm{LE}}=1.1019\,t^{2}=1.59\times10^{-2}c$). A plain polynomial/B\'ezier least-squares fit of $y(x)$ on the sampling rule specified in
Table~\ref{tab:repro} resolves this endpoint singularity only slowly: the maximum
parameter-matched ordinate error remains near the leading edge and follows an empirical
power law with fitted exponent $-0.95$ over the tested degree range. It is still
$2.5\times10^{-3}c$ at degree $24$. Extrapolating that fitted power law, rather than
performing an actual degree-$800$ solve, predicts that a degree of order $800$ would be
needed to reach the $10^{-4}c$ benchmark (Fig.~\ref{fig:conv}, upper curve).

\subsection{The CST decomposition, and what it is not}
The standard aeronautical remedy is the class--shape transformation (CST) of Kulfan
\cite{Kulfan2008}, which the benchmark of \cite{Masters2017} ranks among the most
economical of the common parameterisations, alongside B\'ezier--PARSEC
\cite{DerksenRogalsky2010} and Hicks--Henne \cite{HicksHenne1978}. CST writes
\begin{equation}\label{eq:cst}
y(x)=C(x)\,S(x),\qquad C(x)=x^{1/2}(1-x),
\end{equation}
where the class function $C$ carries the leading-edge square root and the trailing-edge
closure. For each surface the quotient $S=y/C$ has finite one-sided endpoint limits
and is piecewise smooth, with a possible loss of higher-order smoothness at the camber break
$x=\mathfrak p$. Its leading-edge limits are
$S_u(0^+)=+\sqrt{2r_{\mathrm{LE}}}=+0.1781$ and
$S_l(0^+)=-\sqrt{2r_{\mathrm{LE}}}=-0.1781$. One remark is in order, to avoid
an identification we do \emph{not} make: $C$ contains the fractional power $x^{1/2}$, which
is \emph{not} an Askey--Wilson monomial \eqref{eq:phi}: the $\varphi_k$ are defined for
integer $k$ and are polynomials, and the connection theory of \S\ref{sec:main} concerns
integer indices only. The structural kinship between \eqref{eq:cst} and the endpoint-anchored
factorisations of this paper is a motivating analogy, not a theorem; developing a
fractional-index extension of \S\ref{sec:main} would be a separate undertaking.
The direct mathematical link begins only after $S_u$ and $S_l$ have been replaced by
polynomial approximants $P_n\in\Pi_n$. These polynomials admit both an Askey--Wilson
representation and, under \textup{(H1)}, the generalized-power representation of
Theorem~\ref{thm:explicit}. The class function $C$ is not converted and does not enter the
connection coefficients.

\subsection{The experiment}
We fit the shape function $S$ of each surface by a degree-$n$ polynomial $P_n$ and
reconstruct the parameterised ordinates as $y_n=CP_n$. If
$P_n=\sum_{m=0}^n c_m p_m$, the theoretical connection gives
\[
P_n=\sum_{k=0}^n w_k\widetilde\Bs_k^{\,n},
\qquad w_k=\sum_{m=0}^n A_k(n,m)c_m,
\]
with $A_k(n,m)$ from Theorem~\ref{thm:explicit}. The reconstruction $y_n$ is unchanged by
this conversion; only the coordinates of its polynomial factor change. The first two items
below are two bases for exactly the same unweighted least-squares problem at the same
cosine-spaced nodes. The third is a distinct weighted problem at Gauss--Askey--Wilson nodes
and is included to isolate the conditioning supplied by exact discrete orthogonality.
\begin{enumerate}[leftmargin=1.8em,itemsep=2pt]
\item \textbf{Bernstein basis, arbitrary nodes.} Design matrix $V^{B}_{jk}=\Bs^n_k(t_j)$,
where $t_j:=2x_j-1\in[-1,1]$ maps the $M=400$ cosine-spaced chord points onto the
Askey--Wilson interval; the right-hand side is the sampled vector
$S(x_j)=y(x_j)/C(x_j)$. The measured condition numbers $\kappa_2(V^{B})$ of the design
matrix (not of the normal equations, whose condition number would be the square) grow
geometrically, from $1.0\times10^{1}$ at $n=4$ to $6.5\times10^{5}$ at $n=20$ and
$1.0\times10^{7}$ at $n=24$ (Fig.~\ref{fig:cond}). Normalising the columns of $V^{B}$ in the
discrete norm reduces these values by no more than a factor of two (for instance
$5.3\times10^{6}$ instead of $1.0\times10^{7}$ at $n=24$), so the growth reflects the
geometry of the basis and not a mere diagonal scaling.
\item \textbf{Orthonormal Askey--Wilson basis, same nodes.} Design matrix
$V^{AW}_{jm}=\hat p_m(t_j)$, with $\hat p_m$ evaluated by the three-term recurrence and
normalised through Gauss--Askey--Wilson quadrature; parameters as in Table~\ref{tab:repro}.
For these non-Gaussian nodes the Gram matrix is not diagonal and
$\kappa_2(V^{AW})$ is not $1$. Because both design matrices span the same polynomial
space, the two unweighted fits define the same degree-$n$ fitted polynomial in exact
arithmetic; only the coordinate representation and its numerical conditioning differ. The
Askey--Wilson design matrix is, however, dramatically better conditioned and grows only
mildly: measured $5.0\times10^{1}$ at $n=4$, $1.6\times10^{3}$ at $n=20$,
$2.1\times10^{3}$ at $n=24$. The fitted vector $(c_m)$ is exactly the input of the
connection map displayed above.
\item \textbf{Weighted least squares at Gauss nodes.} This is not the same discrete
least-squares objective as in items 1--2: both the nodes and the weights are changed. If the
shape data are sampled at the $N$-point Gauss--Askey--Wilson nodes with their quadrature
weights, $N=n+1$, the rule is exact for all products $\hat p_j\hat p_m$, whose degree is at
most $2n$. Hence $V^{\mathsf T}WV=I$ in exact arithmetic and the orthogonal coefficients are
simply $V^{\mathsf T}W\mathbf s$, where $\mathbf s$ is the sampled shape-data vector. The
measured $\kappa_2(W^{1/2}V)$ is $1$ to ten digits (Gram deviation
$\le4\times10^{-14}$ for $n\le24$). This, and only this, is the precise sense in which the
orthogonal representation yields a perfectly conditioned design matrix.
\end{enumerate}
For the unweighted cosine-node problem of items 1--2, with $n=9$ (ten coefficients
per surface), the reconstructed ordinates approximate the exact NACA~2412 data as summarised
below (Figs.~\ref{fig:airfoil}--\ref{fig:conv}):

\begin{center}
\small
\renewcommand{\arraystretch}{1.2}
\begin{tabular}{@{}lcc@{}}
\hline
Surface & maximum error $/c$ & RMS error $/c$\\
\hline
upper & $4.1\times10^{-5}$ & $1.8\times10^{-5}$\\
lower & $2.4\times10^{-4}$ & $1.0\times10^{-4}$\\
\hline
\end{tabular}
\end{center}

\noindent Thus ten polynomial coefficients per surface place the upper-surface ordinate error within
the prescribed $10^{-4}c$ benchmark tolerance and the lower-surface error within a small
multiple of it. These statements concern parameter-matched ordinates; they do not
assert the same bounds in Euclidean, normal-distance or Hausdorff metrics. The extrapolated
plain representation of $y$ itself would require hundreds of coefficients at this tolerance. This
accords with the general conclusion of \cite{Masters2017} that efficiency statements are
tolerance-dependent; all numbers here refer to this airfoil, these nodes and this tolerance.

\begin{figure}[t]
\centering
\includegraphics[width=\linewidth]{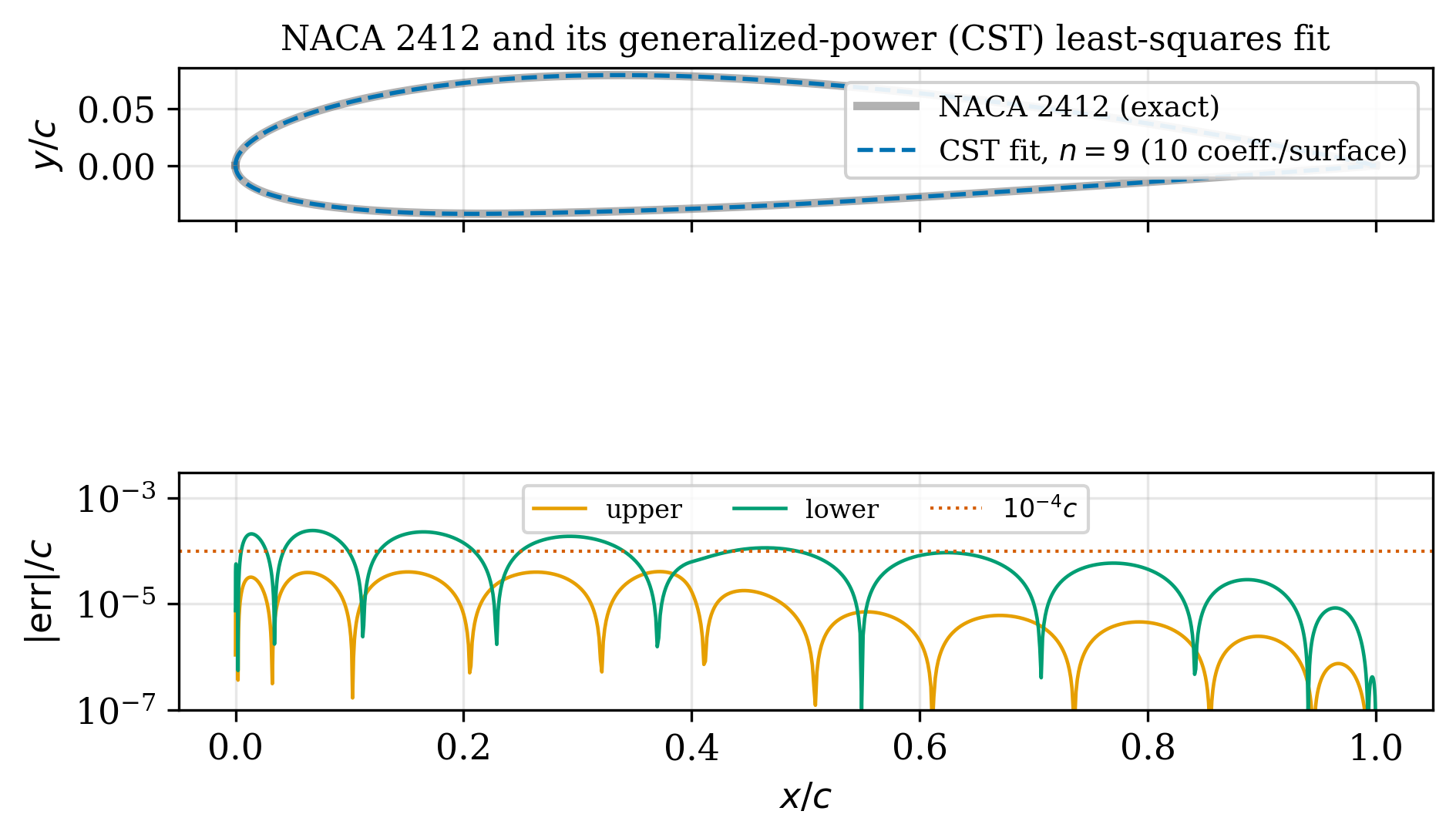}
\caption{NACA~2412 ordinates and their degree-$9$ CST least-squares fits as functions of
the construction parameter $x$ (top), with the parameter-matched absolute ordinate error
for each surface (bottom). Ten coefficients per surface place the upper-surface error within
the prescribed $10^{-4}c$ benchmark tolerance; the largest residuals occur near the leading
edge.}
\label{fig:airfoil}
\end{figure}

\begin{figure}[t]
\centering
\includegraphics[width=0.82\linewidth]{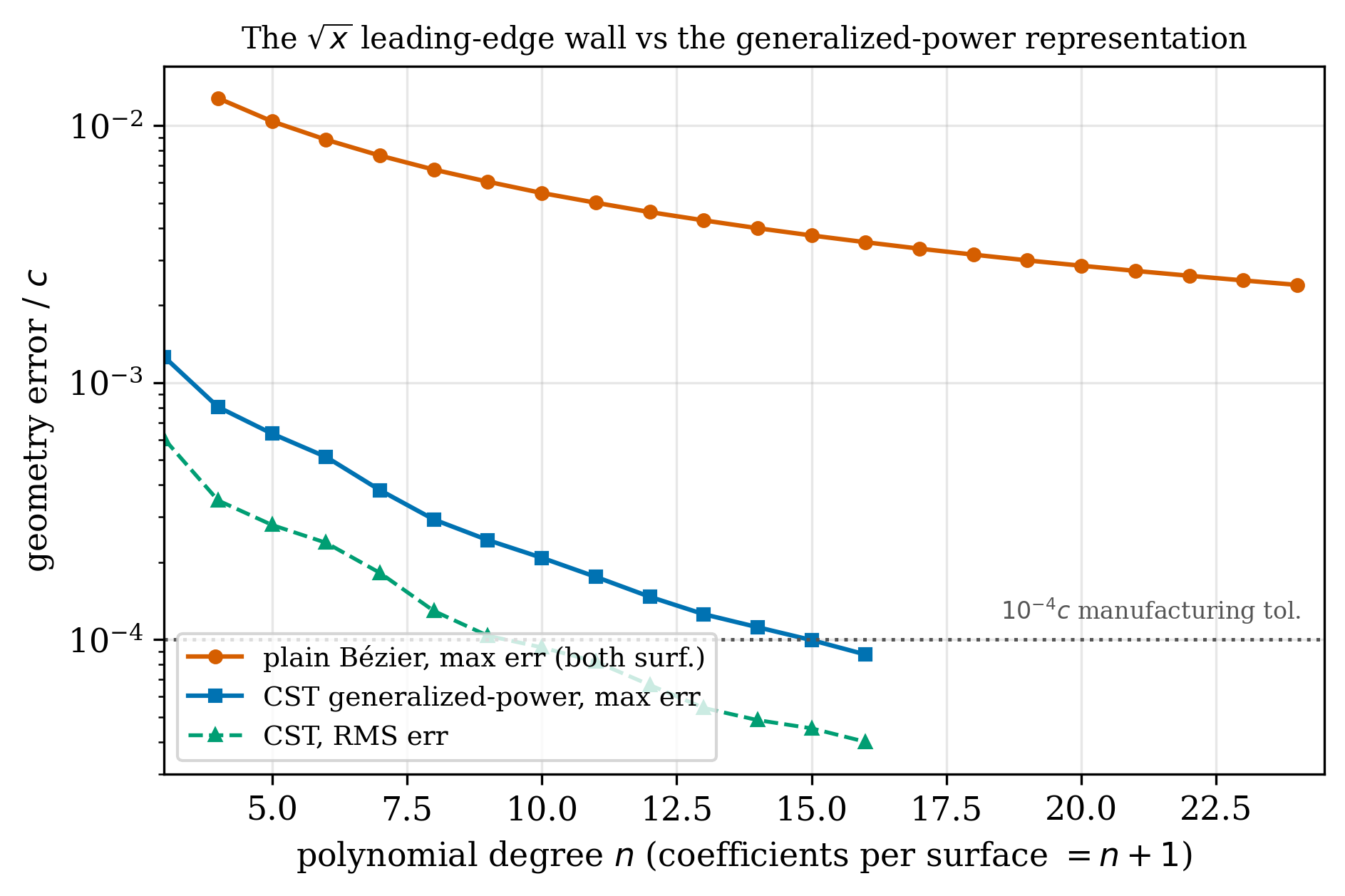}
\caption{Measured parameter-matched ordinate error versus degree for NACA~2412. The
orange and blue curves report the maximum over both surfaces for, respectively, the plain
polynomial/B\'ezier fit of $y$ and the CST reconstruction obtained by fitting $S$; the green
curve reports the corresponding CST RMS error.}
\label{fig:conv}
\end{figure}

\subsection{Exact conversion versus numerical stability}\label{sec:convstab}
Theorem~\ref{thm:explicit} supplies the change of basis
$w_k=\sum_m A_k(n,m)c_m$ from orthogonal coefficients to generalized-power basis
coefficients in closed form. Because the generalized-power basis is not a
nonnegative partition of unity, the numbers $w_k$ are not B\'ezier control ordinates in the
usual convex-hull sense. This conversion is the direct application of the main
connection theorem. By contrast, the affine-Bernstein/Askey--Wilson conditioning comparison
is complementary: $A$ converts to generalized-power coefficients, not to affine Bernstein
coefficients. Two properties must be kept apart.
\emph{Exactness:} because every entry is an explicit ratio of $q$-shifted factorials, the
conversion can be evaluated in exact rational (or arbitrary-precision) arithmetic; for $n=8$, direct substitution of the exact coefficients into the two finite expansions
gives identical polynomials. This degree-$8$ test verifies, at the rational parameters
of Table~\ref{tab:repro}, the same coefficient map that can be applied to any fitted
$P_n\in\Pi_n$. It is intentionally separate from the degree-$9$ NACA accuracy example and
is not an accuracy test of the airfoil fit.
\emph{Stability:} exactness of a formula does not imply a well-conditioned map. Here $A$
denotes the $(n+1)\times(n+1)$ matrix $\big(A_k(n,m)\big)_{k,m}$ of
Theorem~\ref{thm:explicit} itself, in the normalisations of \eqref{eq:AW} and
\eqref{eq:genbern}; its condition number depends on these normalisations, and rescaling
either basis rescales rows or columns accordingly. We measured
$\kappa_2(A)=4.2\times10^{5}$ at $n=4$ and $1.2\times10^{11}$ at $n=8$ (parameters of
Table~\ref{tab:repro}): the normwise first-order condition estimate permits relative errors to be amplified by a
factor as large as $1.2\times10^{11}$ at $n=8$, corresponding to a potential loss of about
eleven decimal digits in the worst case. This is a condition-number bound, not a
claim that every conversion loses eleven digits. It nevertheless shows severe ill-conditioning for the normalization and parameter regime used here.
The practical protocol that follows from these measurements is: fit and store in the
orthogonal basis (representation 2 or 3 above); evaluate geometry by Clenshaw-type recurrence
in that basis; and use the closed form of $A$ (in exact or extended precision, which the
explicit $q$-Pochhammer structure makes inexpensive) only when generalized-power coefficients are genuinely required. The unweighted
least-squares problems were solved with the SVD-based divide-and-conquer routine of LAPACK
(\texttt{DGELSD}). For the weighted Gauss-node formulation, the orthogonal
coefficients were computed as $V^{\mathsf T}W\mathbf s$, since
$V^{\mathsf T}WV=I$ to quadrature accuracy, rather than by a generic least-squares solve, in IEEE double precision except where exact arithmetic is stated.

\begin{figure}[t]
\centering
\includegraphics[width=0.82\linewidth]{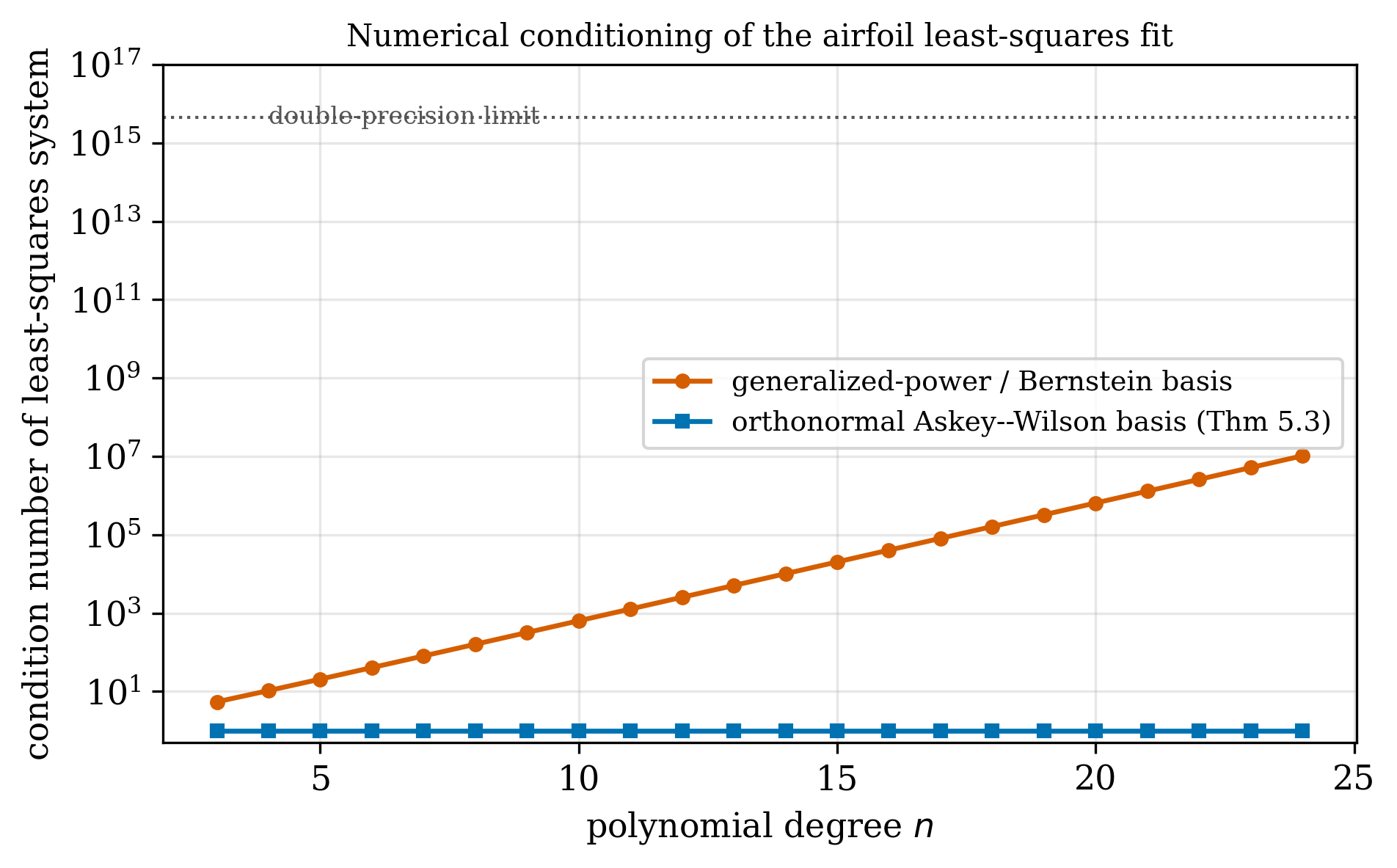}
\caption{Measured $2$-norm condition number of the least-squares design matrix: Bernstein
basis at cosine-spaced nodes (orange), orthonormal Askey--Wilson basis at the same nodes
(blue), and the weighted matrix $W^{1/2}V$ at Gauss--Askey--Wilson nodes (green). In the last
case the Gram matrix is the identity to the reported numerical precision.}
\label{fig:cond}
\end{figure}

\begin{table}[!htb]
\centering\small
\renewcommand{\arraystretch}{1.25}
\begin{tabular}{@{}lp{9.6cm}@{}}
\hline
Geometry & NACA 2412, displayed 4-digit formulas, closed trailing edge ($a_4=-0.1036$)\\
Surfaces & upper and lower ordinates fitted independently as functions of the NACA
construction parameter $x$; $S_{u,l}=y_{u,l}/C$ on $(0,1)$\\
Endpoint values & $S_{u,l}(0^{+})=\pm\sqrt{2r_{\mathrm{LE}}}=\pm0.1781$;
$S_{u,l}(1^{-})=-y_{u,l}'(1)$: $0.2117$ (upper), $-0.0784$ (lower), with derivatives taken
with respect to the construction parameter\\
Fit nodes & $M=400$ points $x_j=\tfrac12(1-\cos\theta_j)$, $\theta_j=j\pi/401$,
$j=1,\dots,400$\\
Variable map & $t_j=2x_j-1\in[-1,1]$ (chord to Askey--Wilson interval)\\
Error grid & $M'=5\,999$ points $x_j=\tfrac12(1-\cos(j\pi/6000))$, $j=1,\dots,5999$;
errors in units of chord $c$\\
Norms & vertical ordinate errors
$\max_j\big|y_{\mathrm{fit}}(x_j)-y_{\mathrm{exact}}(x_j)\big|$ and
$\mathrm{RMS}=\big(\tfrac1{M'}\sum_j\big|y_{\mathrm{fit}}(x_j)-y_{\mathrm{exact}}(x_j)\big|^{2}\big)^{1/2}$
over the error grid\\
AW parameters & $(a,b,c,d;q)=(0.3,\,0.2,\,0.15,\,0.1;\,0.6)$ for the conditioning study\\
 & $(a,b,c,d;q)=(\tfrac3{10},\tfrac15,\tfrac3{20},\tfrac1{10};\tfrac35)$ (exact) for $\kappa_2(A)$ and the conversion\\
Quadrature & Gauss--Askey--Wilson from the AW three-term recurrence (Golub--Welsch);
$N=n+1$ nodes in the weighted study; $N=n+2$ for the normalisations, one node beyond the
$N=n+1$ already sufficient for exactness of the squared norms\\
Solvers & SVD divide-and-conquer (LAPACK \texttt{DGELSD}) for the unweighted LS
problems; $V^{\mathsf T}W\mathbf s$ for the Gauss-node weighted coefficients; exact rational
arithmetic for the basis conversion\\
Precision & IEEE double, except exact-arithmetic statements\\
Software & Python 3.11.15, NumPy 2.4.4, SciPy 1.17.1, SymPy 1.14.0 (reference BLAS/LAPACK)\\
\hline
\end{tabular}
\caption{Reproducibility data for \S\ref{sec:apps}.}
\label{tab:repro}
\end{table}

\subsection{Assessment}
Within its stated scope the experiment supports three conclusions. The leading-edge square-root singularity produces the measured slow convergence of the
plain fit (an empirical $n^{-0.95}$ law over the tested range). The CST factorisation removes
the explicit endpoint singular factor; with ten polynomial coefficients, the upper surface
reaches the benchmark while the lower surface remains at $2.4\times10^{-4}c$.
Accordingly, ``about ten'' coefficients do not place both surfaces inside the
$10^{-4}c$ threshold; they place both within the same order of magnitude. The orthogonal Askey--Wilson representation
improves the conditioning of the same least-squares problem by orders of magnitude at equal
degree, and yields a condition number equal to one, to the stated numerical precision, for the
weighted design matrix in the distinct Gauss-node formulation. The
closed-form connection matrix of Theorem~\ref{thm:explicit} makes the conversion to
generalized-power coefficients exact in exact arithmetic, but it is numerically ill-conditioned, so conversions
should be done in extended precision when needed. What this section does \emph{not} claim:
that CST is an instance of the Askey--Wilson theory (it is not; the class function is outside
$\Pi_n$); that generalized-power coefficients possess the convex-hull interpretation of
classical B\'ezier control ordinates; that parameter-matched ordinate errors are geometric
distance errors; or that any measured constants carry over unchanged to other airfoils,
node distributions, weights or basis normalisations. For the tensor-product and rational (NURBS) surface representations and the
free-form-deformation lattices to which such fits feed in practice, see
\cite{PieglTiller1997,SederbergParry1986,Farin2002,Prautzsch2002}.

\section{Conclusion and outlook}\label{sec:concl}

We have set up the univariate theory of Bernstein-type bases on nonuniform lattices around a
structural dichotomy that is invisible on the linear and $q$-linear lattices. The
partition-of-unity property and the orthogonal-connection property, which rest on a single
falling-factorial basis in the classical and $q$-linear cases, are carried by \emph{two
different} bases on the $q$-quadratic lattice
(Theorem~\ref{thm:dichotomy}): the affine spectral basis \eqref{eq:bern} is the B\'ezier/CAGD
object (Propositions~\ref{prop:bezier}--\ref{prop:decasteljau}), while the first-order ladder has solution-space dimension $\max\{3-n,0\}$ and therefore fails in every degree $n\ge3$ for every $q\in(0,1)$ (Theorem~\ref{prop:ladder}), and the affine factors fail to satisfy the weight-shift mechanism of Proposition~\ref{prop:spectralnotOP}. The stronger nonorthogonality statement is Conjecture~\ref{conj:spectralnotOP}. By contrast, the generalized-power basis
\eqref{eq:genbern} fails to yield a nonnegative partition of unity in the positive
parameter region of Proposition~\ref{prop:obstruction}
(the unique unity-expanding normalisation $\pi_{n,k}$ changes sign there; the normalised family is the Askey--Wilson Bernstein basis of Simeonov and Goldman \cite{SimeonovGoldman2019}, Remark~\ref{rem:SG}), but its
Askey--Wilson connection \emph{is}
orthogonal, with respect to a sign-changing quasi-definite weight
(Theorem~\ref{thm:explicit}, Remark~\ref{rem:orth}). Two mechanisms drive the result: the
weight-shift lemma (Lemma~\ref{lem:wshift}), by which the B\'ezier index enters the
coefficients only through $aq^{k},bq^{n-k}$ (Corollary~\ref{cor:closedform}); and the
bispectral tridiagonal action of the Askey--Wilson operator on the generalized-power basis,
computed here with fully explicit band coefficients \eqref{eq:dk}--\eqref{eq:bk}. The
connection coefficients are $q$-Racah polynomials
$R_m(\mu(k);\,ad/q,\,bc/q,\,q^{-n-1},\,a/b\,|\,q)$ times the fully explicit prefactors
\eqref{eq:prefactor}--\eqref{eq:omega}; the gauged difference equation coincides exactly,
after the parameter identification \eqref{eq:params}, with the standard $q$-Racah difference
equation of \cite[\S14.2]{KLS2010}, and every step of the proof is reduced to finite algebraic identities developed in Appendix~\ref{app:certificates}. The scaled limits of \S\ref{sec:limits} recover the Wilson/Racah connection and the $q$-Hahn polynomial part of the normalised $q$-linear connection coefficients.

Several questions remain open and seem worth pursuing: the resolution of
Conjecture~\ref{conj:spectralnotOP}; the classification of parameter
regions where the quasi-definite orthogonality \eqref{eq:coeffOrth} is positive-definite, and
the indefinite (Pontryagin-type) theory outside them; a Leonard-pair/$AW(3)$ formulation of
the tridiagonal action established in Step~1; stable numerical evaluation of the connection
matrix (cf.\ \S\ref{sec:convstab}); and a fractional-index extension motivated by the CST
class function of \S\ref{sec:apps}. A multivariate extension through the nested-product
(Koornwinder) construction, defining multivariate generalized-power Bernstein polynomials
on the $q$-quadratic simplex and recovering the bivariate $q$-Hahn results of
\cite{LWAG2008}, is a natural further step, although the compatibility of nested anchors, the multivariate
measure and the truncation structure all require separate treatment.

\appendix

\section{Algebraic details for the closed-form identities}\label{app:certificates}

This appendix supplies the finite algebraic calculations invoked in the proof of Theorem~\ref{thm:explicit}. Each calculation follows from the explicitly displayed rational functions by evaluating products of linear factors, cancelling common factors, and applying identities of rational functions. Throughout, $K:=q^{k}$ is treated as an
indeterminate, so that validity of the resulting identities for all integers $0\le k\le n$ is
automatic, and Lemma~\ref{lem:denominators} guarantees that no denominator vanishes.

\subsection{The band coefficients}\label{app:residues}

\begin{proposition}\label{prop:residues}
With the notation of Step~\textup1, the residue quotients defining $\mathfrak d_k$ and
$\mathfrak b_k$ evaluate to \eqref{eq:dk} and \eqref{eq:bk}.
\end{proposition}

\begin{proof}
We treat $\mathfrak d_k$; the case of $\mathfrak b_k$ follows from the symmetry
$\Phi_k(z;a,b)=\Phi_{n-k}(z;b,a)$, which exchanges $T_+\leftrightarrow T_-$ and the two pole
pairs, and amounts to the substitution $(a,k)\leftrightarrow(b,n-k)$ in the final formula.

Write $z_0:=aq^{k-1}$. Among the four denominator factors of $R(z)$ in \eqref{eq:Rquot},
only $1-aq^{k-1}z^{-1}$ vanishes at $z_0$; likewise, among the denominator factors of
$T_-(z)$ in \eqref{eq:Tquot}, only $1-aq^{k-1}z^{-1}$ vanishes there. Both poles being
simple, the residue quotient equals the quotient of the two limits
$\lim_{z\to z_0}(1-z_0/z)G(z)$ and $\lim_{z\to z_0}(1-z_0/z)T_-(z)$, namely
\[
\mathfrak d_k
=A(z_0)\,
\frac{(1-aq^{k}z_0)\big(1-\tfrac aq z_0^{-1}\big)(1-bq^{\,n-k}z_0)\big(1-\tfrac bq z_0^{-1}\big)}
{(1-az_0)(1-bz_0)\big(1-bq^{\,n-k-1}z_0^{-1}\big)}
\cdot
\frac{1-aq^{k-1}z_0}{(1-bq^{\,n-k}z_0)\big(1-bq^{\,n-k}z_0^{-1}\big)} .
\]
Substituting $A(z_0)=\dfrac{(1-az_0)(1-bz_0)(1-cz_0)(1-dz_0)}{(1-z_0^{2})(1-qz_0^{2})}$,
the factors $(1-az_0)(1-bz_0)$ and $(1-bq^{\,n-k}z_0)$ cancel, and with $z_0=aq^{k-1}$ the
remaining factors evaluate as
\begin{gather*}
1-cz_0=1-acq^{k-1},\qquad 1-dz_0=1-adq^{k-1},\\
1-\tfrac aq z_0^{-1}=1-q^{-k},\qquad 1-\tfrac bq z_0^{-1}=1-\tfrac ba q^{-k},\\
1-aq^{k}z_0=1-a^{2}q^{2k-1},\qquad 1-aq^{k-1}z_0=1-a^{2}q^{2k-2},\\
1-z_0^{2}=1-a^{2}q^{2k-2},\qquad 1-qz_0^{2}=1-a^{2}q^{2k-1},\\
1-bq^{\,n-k-1}z_0^{-1}=1-\tfrac ba q^{\,n-2k},\qquad
1-bq^{\,n-k}z_0^{-1}=1-\tfrac ba q^{\,n-2k+1}.
\end{gather*}
The factors $1-a^{2}q^{2k-1}$ and $1-a^{2}q^{2k-2}$ cancel in pairs between numerator and
denominator, leaving
\[
\mathfrak d_k=
\frac{(1-q^{-k})\big(1-\tfrac ba q^{-k}\big)(1-acq^{k-1})(1-adq^{k-1})}
{\big(1-\tfrac ba q^{\,n-2k}\big)\big(1-\tfrac ba q^{\,n-2k+1}\big)} .
\]
{Finally we convert the four factors with negative exponents by the elementary rule
$1-uq^{-t}=-uq^{-t}\big(1-u^{-1}q^{\,t}\big)$, applied once to each:
\begin{align*}
1-q^{-k}&=-q^{-k}\big(1-q^{\,k}\big), &
1-\tfrac ba q^{-k}&=-\tfrac ba q^{-k}\big(1-\tfrac ab q^{\,k}\big),\\
1-\tfrac ba q^{\,n-2k}&=-\tfrac ba q^{\,n-2k}\big(1-\tfrac ab q^{\,2k-n}\big), &
1-\tfrac ba q^{\,n-2k+1}&=-\tfrac ba q^{\,n-2k+1}\big(1-\tfrac ab q^{\,2k-n-1}\big).
\end{align*}
The four sign factors $(-1)$ cancel in pairs between numerator and denominator, and the
accumulated unit is
\[
\frac{q^{-k}\cdot\tfrac ba q^{-k}}
{\tfrac ba q^{\,n-2k}\cdot\tfrac ba q^{\,n-2k+1}}
=\frac ab\,q^{\,2k-2n-1},
\]
which, together with the four converted positive-exponent factors, is exactly
\eqref{eq:dk}. As for $\mathfrak b_k$, the symmetry
$\Phi_k(z;a,b)=\Phi_{n-k}(z;b,a)$, combined with the invariance of $A(z)$ under
$a\leftrightarrow b$, gives $H(z;a,b,k)=H(z;b,a,n-k)$ and
$T_{+}(z;a,b,k)=T_{-}(z;b,a,n-k)$; consequently
$\mathfrak b_k(a,b)=\mathfrak d_{n-k}(b,a)$, and applying the substitution
$(a,k)\mapsto(b,n-k)$, $(b,n-k)\mapsto(a,k)$ to \eqref{eq:dk}, under which
$q^{\,2k-2n-1}\mapsto q^{-2k-1}$, $1-q^{\,k}\mapsto1-q^{\,n-k}$,
$1-\tfrac ab q^{\,k}\mapsto1-\tfrac ba q^{\,n-k}$,
$1-acq^{\,k-1}\mapsto1-bcq^{\,n-k-1}$, $1-adq^{\,k-1}\mapsto1-bdq^{\,n-k-1}$ and the
denominator pair maps to
$\big(1-\tfrac ba q^{\,n-2k-1}\big)\big(1-\tfrac ba q^{\,n-2k}\big)$, yields
\eqref{eq:bk}.}

For the values at infinity recorded in \eqref{eq:infvalues}: from \eqref{eq:Lform} and
\eqref{eq:Rquot} one reads off $A(\infty)=abcd/q$, $A(0)=1$, $R(\infty)=q^{\,n}$ and
$R(0)=q^{-n}$ (in each case the four linear factors contribute their leading coefficients),
so that $H(\infty)=A(\infty)(R(\infty)-1)+A(0)(R(0)-1)$ is the stated value; the limits of
$T_{\pm}$ are immediate from \eqref{eq:Tquot}.
\end{proof}

\subsection{The gauged band}\label{app:BD}

\begin{proposition}\label{prop:BD}
With $\rho_k$ as in \eqref{eq:rho}, the products \eqref{eq:BDdef} simplify to
\eqref{eq:Bq} and \eqref{eq:Dq}.
\end{proposition}

\begin{proof}
Again we display the algebraic reduction for $B(k)=\mathfrak d_{k+1}\rho_k$; the reduction of $D(k)$ has the same structure. Substituting $k+1$ for $k$ in \eqref{eq:dk},
\[
\mathfrak d_{k+1}=\frac ab\,q^{\,2k-2n+1}\,
\frac{(1-q^{\,k+1})\big(1-\tfrac ab q^{\,k+1}\big)(1-acq^{\,k})(1-adq^{\,k})}
{\big(1-\tfrac ab q^{\,2k-n+1}\big)\big(1-\tfrac ab q^{\,2k-n+2}\big)} .
\]
Multiplying by \eqref{eq:rho}, the factors $1-q^{\,k+1}$, $1-\tfrac ab q^{\,k+1}$ and
$1-\tfrac ab q^{\,2k+2-n}$ cancel, and there remains
\[
B(k)=\frac ab\,q^{\,2k-2n}\,
\frac{(1-acq^{\,k})(1-adq^{\,k})(1-q^{\,n-k})\big(1-\tfrac ba q^{\,n-k}\big)}
{\big(1-\tfrac ab q^{\,2k-n+1}\big)\big(1-\tfrac ab q^{\,2k-n}\big)} .
\]
It remains to translate into the parameters \eqref{eq:params}. One has
\begin{gather*}
1-acq^{\,k}=1-\beta\delta q^{\,k+1},\qquad 1-adq^{\,k}=1-\alpha q^{\,k+1},\\
1-q^{\,n-k}=-q^{\,n-k}\big(1-\gamma q^{\,k+1}\big),\qquad
1-\tfrac ba q^{\,n-k}=-\tfrac ba q^{\,n-k}\big(1-\gamma\delta q^{\,k+1}\big),
\end{gather*}
while the denominators are $1-\gamma\delta q^{\,2k+2}$ and $1-\gamma\delta q^{\,2k+1}$. The accumulated unit is
$\tfrac ab\,q^{\,2k-2n}\cdot q^{\,n-k}\cdot\tfrac ba q^{\,n-k}=1$, and \eqref{eq:Bq}
follows. The same procedure applied to $\mathfrak b_{k-1}/\rho_{k-1}$ yields, after the
analogous cancellations and the conversions
$\beta-\gamma q^{\,k}=\beta\big(1-\tfrac{q^{\,k-n}}{bc}\big)$ and
$\alpha-\gamma\delta q^{\,k}=\alpha\big(1-\tfrac{q^{\,k-n}}{bd}\big)$, the expression
\eqref{eq:Dq}.
\end{proof}

\subsection{The diagonal identity}\label{app:diag}

\begin{proposition}\label{prop:diag}
For $0\le k\le n$ one has $\mathfrak a_k=-\big(B(k)+D(k)\big)$.
\end{proposition}

\begin{proof}
By Step~\textup1 and \eqref{eq:infvalues},
\[
-\mathfrak a_k=\mathfrak b_k\,T_+(\infty)+\mathfrak d_k\,T_-(\infty)-H(\infty)
=\beta_\ast(K)+\delta_\ast(K)-H(\infty),
\]
where, multiplying \eqref{eq:bk} and \eqref{eq:dk} by the limits \eqref{eq:infvalues} and
writing $K=q^{k}$,
\[
\beta_\ast(K)=q^{-n}\,
\frac{(1-q^{\,n-k})\big(1-\tfrac ba q^{\,n-k}\big)(1-bcq^{\,n-k-1})(1-bdq^{\,n-k-1})}
{\big(1-\tfrac ba q^{\,n-2k}\big)\big(1-\tfrac ba q^{\,n-2k-1}\big)},
\]
\[
\delta_\ast(K)=q^{-n}\,
\frac{(1-q^{\,k})\big(1-\tfrac ab q^{\,k}\big)(1-acq^{\,k-1})(1-adq^{\,k-1})}
{\big(1-\tfrac ab q^{\,2k-n-1}\big)\big(1-\tfrac ab q^{\,2k-n}\big)} .
\]
Consider therefore
\[
\Theta(K):=\beta_\ast(K)+\delta_\ast(K)-H(\infty)-B(K)-D(K),
\]
a rational function of $K$. {Written in the variable $K$, the six candidate poles are the
two square roots $K_0=\pm K_\ast$ of each of the three values
$K_\ast^{2}\in\big\{\tfrac ba q^{\,n-1},\ \tfrac ba q^{\,n},\ \tfrac ba q^{\,n+1}\big\}$;
under \textup{(H)} they are simple, the first value occurring in $\beta_\ast$ and in $B$,
the third in $\delta_\ast$ and in $D$, and the middle value in all four terms. For each of the
three pole loci, the residue identity below is derived using only the defining quadratic
relation for $K_0$, so the same computation applies verbatim to both choices of square
root. Two elementary observations recur. First, at a common
zero $K_0$ of a pair of factors of the two shapes $1-\tfrac ba q^{\,s}K^{-2}$ and
$1-\tfrac ab q^{-s}K^{2}$, the linearisations are
$\tfrac2{K_0}(K-K_0)$ and $-\tfrac2{K_0}(K-K_0)$ respectively, so the residue of a term is
$\pm\tfrac{K_0}2$ times the remaining factors evaluated at $K_0$, with the sign determined
by which shape carries the pole. Second, on the pole locus one may substitute
$K_0^{-1}=\tfrac ab q^{-s}K_0$, which converts every factor with a negative power of $K_0$
into one with a positive power.

\emph{The pole $K^{2}=\tfrac ba q^{\,n-1}$.} Here the vanishing factors are
$1-\tfrac ba q^{\,n-1}K^{-2}$ in $\beta_\ast$ and $1-\tfrac ab q^{\,1-n}K^{2}$ in $B$, of
opposite shapes, so the residues of $\beta_\ast$ and of $-B$ in $\Theta$ add to
$\tfrac{K_0}2\big([\beta_\ast]'+[B]'\big)$, where the primes denote the remaining factors at
$K_0$. Substituting $K_0^{-1}=\tfrac ab q^{\,1-n}K_0$ into the four numerator factors of
$\beta_\ast$ gives
\begin{gather*}
1-\tfrac{q^{\,n}}{K_0}=1-\tfrac ab qK_0,\qquad
1-\tfrac ba\tfrac{q^{\,n}}{K_0}=1-qK_0,\\
1-\tfrac{bcq^{\,n-1}}{K_0}=1-acK_0,\qquad
1-\tfrac{bdq^{\,n-1}}{K_0}=1-adK_0,
\end{gather*}
while the surviving denominators are $1-\tfrac ba q^{\,n}K_0^{-2}=1-q$ in $\beta_\ast$ and
$1-\tfrac ab q^{-n}K_0^{2}=1-q^{-1}=-q^{-1}(1-q)$ in $B$. Hence
\begin{gather*}
[\beta_\ast]'=\frac{q^{-n}\big(1-\tfrac ab qK_0\big)(1-qK_0)(1-acK_0)(1-adK_0)}{1-q},\\
[B]'=-\,\frac{q\,(1-adK_0)(1-acK_0)\big(1-q^{-n}K_0\big)\big(1-\tfrac ab q^{-n}K_0\big)}{1-q},
\end{gather*}
and after cancelling $(1-acK_0)(1-adK_0)$ the required identity
$[\beta_\ast]'+[B]'=0$ reduces to
\[
q^{-n}\big(1-\tfrac ab qK_0\big)(1-qK_0)
=q\big(1-q^{-n}K_0\big)\big(1-\tfrac ab q^{-n}K_0\big).
\]
Expanding both sides and using $K_0^{2}=\tfrac ba q^{\,n-1}$, each side equals
$q^{-n}+q-q^{\,1-n}\big(1+\tfrac ab\big)K_0$, which proves the cancellation.

\emph{The pole $K^{2}=\tfrac ba q^{\,n+1}$.} Only $\delta_\ast$ and $D$ are singular, both
through the factor $1-\tfrac ab q^{-n-1}K^{2}$, of the same shape, so the residues of
$\delta_\ast$ and of $-D$ cancel precisely when $[\delta_\ast]'=[D]'$. The surviving
denominators both equal $1-\tfrac ab q^{-n}K_0^{2}=1-q$, and the factors
$(1-K_0)\big(1-\tfrac ab K_0\big)$ are common, so the identity reduces to
\[
q^{-n}\Big(1-\tfrac{ac}{q}K_0\Big)\Big(1-\tfrac{ad}{q}K_0\Big)
=q\Big(\tfrac{bc}q-q^{-n-1}K_0\Big)\Big(\tfrac{ad}q-\tfrac ab q^{-n-1}K_0\Big).
\]
Converting the two right-hand factors by
$1-q^{-n}K_0/(bc)=-\,\tfrac{q^{-n}K_0}{bc}\big(1-\tfrac{ac}qK_0\big)$ and
$1-q^{-n}K_0/(bd)=-\,\tfrac{q^{-n}K_0}{bd}\big(1-\tfrac{ad}qK_0\big)$, both consequences of
$K_0^{-1}=\tfrac ab q^{-n-1}K_0$, the right-hand side becomes
$\tfrac{abcd}q\cdot\tfrac{q^{-2n}K_0^{2}}{b^{2}cd}\big(1-\tfrac{ac}qK_0\big)
\big(1-\tfrac{ad}qK_0\big)=q^{-n}\big(1-\tfrac{ac}qK_0\big)\big(1-\tfrac{ad}qK_0\big)$,
using $K_0^{2}=\tfrac ba q^{\,n+1}$, which is the left-hand side.

\emph{The shared pole $K^{2}=\tfrac ba q^{\,n}$.} All four terms are singular:
$\beta_\ast$ through $1-\tfrac ba q^{\,n}K^{-2}$ and $\delta_\ast$, $B$, $D$ through
$1-\tfrac ab q^{-n}K^{2}$. By the sign rule above, the residue of $\Theta$ is
$\tfrac{K_0}2\big([\beta_\ast]'-[\delta_\ast]'+[B]'+[D]'\big)$. Substituting
$K_0^{-1}=\tfrac ab q^{-n}K_0$ into the numerator factors of $\beta_\ast$ yields
$1-\tfrac{q^{\,n}}{K_0}=1-\tfrac ab K_0$, $1-\tfrac ba\tfrac{q^{\,n}}{K_0}=1-K_0$,
$1-\tfrac{bcq^{\,n-1}}{K_0}=1-\tfrac{ac}qK_0$ and
$1-\tfrac{bdq^{\,n-1}}{K_0}=1-\tfrac{ad}qK_0$, while the surviving denominator is
$1-\tfrac ba q^{\,n-1}K_0^{-2}=1-q^{-1}$; comparison with $\delta_\ast$, whose surviving
denominator is $1-\tfrac ab q^{-n-1}K_0^{2}=1-q^{-1}$ as well, shows that
$[\beta_\ast]'=[\delta_\ast]'$, so the first pair cancels outright. For the second pair,
the same substitution converts the two rightmost numerator factors of $B$,
$1-q^{-n}K_0=-q^{-n}K_0\big(1-\tfrac ab K_0\big)$ and
$1-\tfrac ab q^{-n}K_0=-\tfrac ab q^{-n}K_0(1-K_0)$, giving
\[
[B]'=\frac{q^{-n}\,(1-adK_0)(1-acK_0)(1-K_0)\big(1-\tfrac ab K_0\big)}{1-q},
\]
while in $D$ one writes
$\tfrac{bc}q-q^{-n-1}K_0=\tfrac{bc}q\big(1-\tfrac{q^{-n}K_0}{bc}\big)$ and
$\tfrac{ad}q-\tfrac ab q^{-n-1}K_0=\tfrac{ad}q\big(1-\tfrac{q^{-n}K_0}{bd}\big)$ and applies
the conversions $1-\tfrac{q^{-n}K_0}{bc}=-\tfrac{q^{-n}K_0}{bc}(1-acK_0)$ and
$1-\tfrac{q^{-n}K_0}{bd}=-\tfrac{q^{-n}K_0}{bd}(1-adK_0)$, now consequences of
$K_0^{-1}=\tfrac ab q^{-n}K_0$; collecting the units, one finds
\[
[D]'=-\,\frac{q^{-n}\,(1-acK_0)(1-adK_0)(1-K_0)\big(1-\tfrac ab K_0\big)}{1-q}
=-[B]' .
\]
Hence the four residues cancel and $\Theta$ has no poles.} As $K\to\infty$, the four limits are read off from the displayed formulas,
\[
\beta_\ast\to q^{-n},\qquad \delta_\ast\to \frac{abcd\,q^{\,n}}{q},\qquad
B\to\frac{abcd}{q},\qquad D\to1,
\]
so that
$\Theta(\infty)=q^{-n}+\tfrac{abcd}q q^{\,n}-\big(\tfrac{abcd}q(q^{\,n}-1)+q^{-n}-1\big)
-\tfrac{abcd}q-1=0$. A pole-free rational function with limit zero at infinity vanishes
identically, which proves the assertion.
\end{proof}

\subsection{The telescoped gauge}\label{app:telescope}

\begin{proposition}\label{prop:telescope}
The product $\widehat c_k=\prod_{j=0}^{k-1}\rho_j$ equals \eqref{eq:chat}.
\end{proposition}

\begin{proof}
The five groupings displayed in Step~\textup6 telescope separately: the quotient of
$q$-numbers gives the $q$-binomial coefficient; the factors
$1-\tfrac ba q^{\,n-j}$ accumulate to $\ph{\tfrac ba q^{\,n-k+1}}{k}$; the two-step quotient
$\big(1-\tfrac ab q^{\,2j+2-n}\big)\big/\big(1-\tfrac ab q^{\,2j-n}\big)$ telescopes along
the even sublattice to
$\big(1-\tfrac ab q^{\,2k-n}\big)\big/\big(1-\tfrac ab q^{-n}\big)$; and the remaining two
groupings contribute $q^{-k}$ and $\ph{\tfrac ab q}{k}^{-1}$. Hence
\[
\widehat c_k=q^{-k}\,\qbinom nk\,
\frac{\ph{\tfrac ba q^{\,n-k+1}}{k}\,\big(1-\tfrac ab q^{\,2k-n}\big)}
{\big(1-\tfrac ab q^{-n}\big)\,\ph{\tfrac ab q}{k}} .
\]
Two elementary conversions finish the proof. First,
$1-\tfrac ab q^{\,2k-n}=-\tfrac ab q^{\,2k-n}\big(1-\tfrac ba q^{\,n-2k}\big)$ and
$1-\tfrac ab q^{-n}=-\tfrac ab q^{-n}\big(1-\tfrac ba q^{\,n}\big)$, so their quotient is
$q^{\,2k}\big(1-\tfrac ba q^{\,n-2k}\big)\big/\big(1-\tfrac ba q^{\,n}\big)$. Second, by
splitting $\ph{b/a}{n}=\ph{b/a}{n-k}\,\ph{\tfrac ba q^{\,n-k}}{k}$ one has
\[
\frac{\ph{\tfrac ba q^{\,n-k+1}}{k}}{\ph{b/a}{n}}
=\frac{1-\tfrac ba q^{\,n}}{\big(1-\tfrac ba q^{\,n-k}\big)\,\ph{b/a}{n-k}} .
\]
Combining the two displays and multiplying by $\ph{ab}{n}\ph{b/a}{n}$ yields exactly
\eqref{eq:chat}.
\end{proof}

\section*{Acknowledgments}

The author thanks to Qiyuan (Alex) Gu from University of Chicago for pointing out an argument to solve a previous conjecture in v1.


\begin{thebibliography}{99}
\small

\bibitem{AskeyWilson1979} R.~Askey, J.~Wilson,
\emph{A set of orthogonal polynomials that generalize the Racah coefficients or $6$-$j$
symbols}, SIAM J.\ Math.\ Anal.\ \textbf{10} (1979), no.~5, 1008--1016.
DOI: 10.1137/0510092.

\bibitem{RAGZ1998} A.~Ronveaux, A.~Zarzo, I.~Area, E.~Godoy,
\emph{Bernstein bases and Hahn--Eberlein orthogonal polynomials}, Integral Transforms Spec.\
Funct.\ \textbf{7} (1998), no.~1--2, 87--96. DOI: 10.1080/10652469808819188.

\bibitem{ARG2004} I.~Area, E.~Godoy, P.~Wo\'zny, S.~Lewanowicz, A.~Ronveaux,
\emph{Formulae relating little $q$-Jacobi, $q$-Hahn and $q$-Bernstein polynomials:
application to $q$-B\'ezier curve evaluation}, Integral Transforms Spec.\ Funct.\ \textbf{15}
(2004), no.~5, 375--385. DOI: 10.1080/10652460410001727491.

\bibitem{LWAG2008} S.~Lewanowicz, P.~Wo\'zny, I.~Area, E.~Godoy,
\emph{Multivariate generalized Bernstein polynomials: identities for orthogonal polynomials
of two variables}, Numer.\ Algorithms \textbf{49} (2008), 199--220.
DOI: 10.1007/s11075-008-9168-9.

\bibitem{Phillips1997} G.~M.~Phillips,
\emph{Bernstein polynomials based on the $q$-integers}, Ann.\ Numer.\ Math.\ \textbf{4}
(1997), 511--518.

\bibitem{OrucPhillips2003} H.~Oru\c{c}, G.~M.~Phillips,
\emph{$q$-Bernstein polynomials and B\'ezier curves}, J.\ Comput.\ Appl.\ Math.\ \textbf{151}
(2003), 1--12. DOI: 10.1016/S0377-0427(02)00733-1.

\bibitem{QinHu2013} X.~Qin, G.~Hu, N.~Zhang, X.~Shen, Y.~Yang,
\emph{A novel extension to the polynomial basis functions describing B\'ezier curves and
surfaces of degree $n$ with multiple shape parameters}, Appl.\ Math.\ Comput.\ \textbf{223}
(2013), 1--16. DOI: 10.1016/j.amc.2013.07.073.

\bibitem{DelgadoPena2020} J.~Delgado, J.~M.~Pe\~na,
\emph{Geometric properties and algorithms for rational $q$-B\'ezier curves and surfaces},
Mathematics \textbf{8} (2020), no.~4, art.~541. DOI: 10.3390/math8040541.

\bibitem{KLS2010} R.~Koekoek, P.~A.~Lesky, R.~F.~Swarttouw,
\emph{Hypergeometric Orthogonal Polynomials and Their $q$-Analogues}, Springer Monographs in
Mathematics, Springer, Berlin, 2010. DOI: 10.1007/978-3-642-05014-5.

\bibitem{Ismail2005} M.~E.~H.~Ismail,
\emph{Classical and Quantum Orthogonal Polynomials in One Variable}, Encyclopedia Math.\
Appl.\ \textbf{98}, Cambridge Univ.\ Press, Cambridge, 2005 (hardback edition).
DOI: 10.1017/CBO9781107325982.

\bibitem{GasperRahman2004} G.~Gasper, M.~Rahman,
\emph{Basic Hypergeometric Series}, 2nd ed., Encyclopedia Math.\ Appl.\ \textbf{96},
Cambridge Univ.\ Press, Cambridge, 2004.

\bibitem{DLMF} F.~W.~J.~Olver, A.~B.~Olde Daalhuis, D.~W.~Lozier, B.~I.~Schneider,
R.~F.~Boisvert, C.~W.~Clark, B.~R.~Miller, B.~V.~Saunders, H.~S.~Cohl, M.~A.~McClain, eds.,
\emph{NIST Digital Library of Mathematical Functions}, Release~1.2.4 of 2025,
\S\S18.25 and 18.28, \url{https://dlmf.nist.gov/} (accessed 27~July 2026).

\bibitem{NSU1991} A.~F.~Nikiforov, S.~K.~Suslov, V.~B.~Uvarov,
\emph{Classical Orthogonal Polynomials of a Discrete Variable}, Springer Series in
Computational Physics, Springer, Berlin, 1991.

\bibitem{Foup2008} M.~Foupouagnigni,
\emph{On difference equations for orthogonal polynomials on nonuniform lattices},
J.\ Difference Equ.\ Appl.\ \textbf{14} (2008), no.~2, 127--174.
DOI: 10.1080/10236190701536199.

\bibitem{FKKM2013} M.~Foupouagnigni, W.~Koepf, M.~Kenfack-Nangho, S.~Mboutngam,
\emph{On solutions of holonomic divided-difference equations on nonuniform lattices},
Axioms \textbf{2} (2013), no.~3, 404--434. DOI: 10.3390/axioms2030404.

\bibitem{FoupMbout2019} M.~Foupouagnigni, S.~Mboutngam,
\emph{On the polynomial solution of divided-difference equations of the hypergeometric type
on nonuniform lattices}, Axioms \textbf{8} (2019), no.~2, art.~47.
DOI: 10.3390/axioms8020047.

\bibitem{Farin2002} G.~Farin,
\emph{Curves and Surfaces for CAGD: A Practical Guide}, 5th ed., Morgan Kaufmann, 2002.

\bibitem{Prautzsch2002} H.~Prautzsch, W.~Boehm, M.~Paluszny,
\emph{B\'ezier and B-Spline Techniques}, Springer, Berlin--Heidelberg, 2002.
DOI: 10.1007/978-3-662-04919-8.

\bibitem{Kulfan2008} B.~M.~Kulfan,
\emph{Universal parametric geometry representation method}, J.\ Aircraft \textbf{45} (2008),
no.~1, 142--158. DOI: 10.2514/1.29958.

\bibitem{DerksenRogalsky2010} R.~W.~Derksen, T.~Rogalsky,
\emph{B\'ezier-PARSEC: an optimized aerofoil parameterization for design}, Adv.\ Eng.\
Software \textbf{41} (2010), no.~7--8, 923--930. DOI: 10.1016/j.advengsoft.2010.05.002.

\bibitem{Masters2017} D.~A.~Masters, N.~J.~Taylor, T.~C.~S.~Rendall, C.~B.~Allen, D.~J.~Poole,
\emph{Geometric comparison of aerofoil shape parameterization methods}, AIAA J.\ \textbf{55}
(2017), no.~5, 1575--1589. DOI: 10.2514/1.J054943.

\bibitem{HicksHenne1978} R.~M.~Hicks, P.~A.~Henne,
\emph{Wing design by numerical optimization}, J.\ Aircraft \textbf{15} (1978), no.~7,
407--412. DOI: 10.2514/3.58379.

\bibitem{PieglTiller1997} L.~Piegl, W.~Tiller,
\emph{The NURBS Book}, 2nd ed., Springer, Berlin, 1997. DOI: 10.1007/978-3-642-59223-2.

\bibitem{SimeonovGoldman2019} P.~Simeonov, R.~Goldman,
\emph{A polynomial blossom for the Askey--Wilson operator}, Constr.\ Approx.\ \textbf{50}
(2019), 19--43. DOI: 10.1007/s00365-018-9439-1.

\bibitem{SimeonovGoldman2024} P.~Simeonov, R.~Goldman,
\emph{Rational Askey--Wilson Bernstein bases and a multirational Askey--Wilson blossom},
J.\ Comput.\ Algebra \textbf{12} (2024), 100025. DOI: 10.1016/j.jaca.2024.100025.

\bibitem{SederbergParry1986} T.~W.~Sederberg, S.~R.~Parry,
\emph{Free-form deformation of solid geometric models}, Comput.\ Graph.\ (SIGGRAPH~'86)
\textbf{20} (1986), no.~4, 151--160. DOI: 10.1145/15886.15903.

\end{thebibliography}
\end{document}